\documentclass[final,3p,cmr10]{elsarticle}
\usepackage{moreverb}
\usepackage{amsthm}
\usepackage[T1]{fontenc}
\usepackage{color}
\usepackage{array}
\usepackage{textcomp}
\usepackage{amsthm}
\usepackage{amsbsy}
\usepackage{amstext}
\usepackage{amssymb}
\usepackage{graphicx}
\usepackage{setspace}
\usepackage{arydshln}
\usepackage{threeparttable}
\usepackage[linesnumbered,ruled,vlined]{algorithm2e}
\usepackage[unicode,bookmarksopen,bookmarksnumbered,citecolor=black,urlcolor=black]{hyperref}
\usepackage{autonum}
\usepackage{enumerate}
\newtheorem{theorem}{\protect\theoremname}
\newtheorem{definition}[theorem]{\protect\definitionname}
\newtheorem{proposition}[theorem]{\protect\propositionname}
\newtheorem{lemma}[theorem]{\protect\lemmaname}
\newtheorem{corollary}[theorem]{\protect\corollaryname}
\newtheorem{assumption}[theorem]{Assumption}

\providecommand{\corollaryname}{Corollary}
\providecommand{\lemmaname}{Lemma}
\providecommand{\propositionname}{Proposition}
\providecommand{\remarkname}{Remark}
\providecommand{\definitionname}{Definition}
\providecommand{\theoremname}{Theorem}

\journal{arXiv}

\begin{document}

\begin{frontmatter}

\title{Dimension-wise Multivariate Orthogonal Polynomials in General Probability Spaces\tnoteref{t1}}
\author{Sharif Rahman\fnref{fn1}}
\address{Applied Mathematical and Computational Sciences, The University of Iowa, Iowa City, Iowa 52242, U.S.A.}
\ead{sharif-rahman@uiowa.edu}
\fntext[fn1]{Professor.}
\tnotetext[t1]{Grant sponsor: U.S. National Science Foundation; Grant No. CMMI-1462385.}

\begin{abstract}
This paper puts forward a new generalized polynomial dimensional decomposition (PDD), referred to as GPDD, comprising hierarchically ordered measure-consistent multivariate orthogonal polynomials in dependent random variables.  Unlike the existing PDD, which is valid strictly for independent random variables, no tensor-product structure is assumed or required.  Important mathematical properties of GPDD are studied by constructing dimension-wise decomposition of polynomial spaces, deriving statistical properties of random orthogonal polynomials, demonstrating completeness of orthogonal polynomials for prescribed assumptions, and proving mean-square convergence to the correct limit, including when there are infinitely many random variables.  The GPDD approximation proposed should be effective in solving high-dimensional stochastic problems subject to dependent variables.
\end{abstract}

\begin{keyword}
Uncertainty quantification, multivariate orthogonal polynomials, ANOVA, polynomial dimensional decomposition, non-product-type probability measures.
\end{keyword}

\end{frontmatter}

\section{Introduction}
Polynomial dimensional decomposition (PDD) is a dimension-wise Fourier series expansion in random orthogonal polynomials \cite{rahman08,rahman18}.  Methods rooted in PDD are commonly used to solve high-dimensional stochastic problems in many fields, such as solid mechanics \cite{chakraborty08}, structural dynamics \cite{rahman11}, fluid dynamics \cite{tang16}, and design optimization \cite{ren16}. The decomposition, introduced by the author as the polynomial variant of the analysis-of-variance (ANOVA) dimensional decomposition (ADD) \cite{griebel10,hoeffding48, kuo10, rabitz99, rahman14,sobol93}, mitigates the curse of dimensionality to some extent by building an input-output behavior of complex systems with low effective dimensions \cite{caflisch97}.  However, the existing PDD is largely founded on the independence assumption of input variables.  The assumption exploits product-type probability measures, facilitating construction of the space of multivariate orthogonal polynomials via the tensor product of the spaces of univariate orthogonal polynomials.  In reality, there may exist significant correlation or dependence among input variables, impeding or invalidating most methods, including the existing PDD.  The Rosenblatt transformation \cite{rosenblatt52}, commonly used for mapping dependent to independent variables, may induce very strong nonlinearity to a stochastic response, potentially degrading or even prohibiting convergence of probabilistic solutions \cite{rahman09}. Therefore, innovations beyond tensor-product PDD, capable of tackling non-product-type probability measures, are highly desirable.

This paper proposes a new generalized PDD, referred to as GPDD, to account for dependent, non-product-type probability measures of input random variables. Mathematical issues concerning necessary and sufficient conditions for the completeness of multivariate orthogonal polynomials, statistical properties of orthogonal polynomials, convergence, and truncations -- all associated with the GPDD -- are studied. The paper is structured as follows.  Section 2 defines or discusses mathematical notations and preliminaries.  The assumptions on the input probability measure are clarified. Section 3 briefly reviews the generalized ADD for dependent variables, providing a vital link to the development of GPDD. An exposition of multivariate orthogonal polynomials consistent with a general, non-product-type probability measure, including derivation of their second moment properties, is given in Section 4.  The section also describes relevant polynomial spaces and construction of their orthogonal decompositions.  The orthogonal basis and completeness of multivariate orthogonal polynomials have also been established.  Section 5 formally introduces GPDD for a square-integrable output random variable, including discussion on its convergence, exactness, and optimality.  In the same section, the approximation quality of a truncated GPDD is discussed. The section ends with an explanation on when the GPDD proposed can be extended for infinitely many input variables.  Section 6 presents a numerical example and discusses the effectiveness of GPDD.  Finally, conclusions are drawn in Section 7.


\section{Input random variables}
Let $\mathbb{N}:=\{1,2,\ldots\}$, $\mathbb{N}_{0}:=\mathbb{N} \cup \{0\}$, $\mathbb{R}:=(-\infty,+\infty)$, and $\mathbb{R}_{0}^{+}:=[0,+\infty)$ represent the sets of positive integer (natural), non-negative integer, real, and non-negative real numbers, respectively. For a non-zero, finite integer $N\in\mathbb{N}$,
denote by $\mathbb{A}^N \subseteq \mathbb{R}^N$ a bounded or unbounded subdomain of $\mathbb{R}^N$.

Let $(\Omega,\mathcal{F},\mathbb{P})$ be a complete probability space, where $\Omega$ is a sample space representing an abstract set of elementary events, $\mathcal{F}$ is a $\sigma$-algebra on $\Omega$, and $\mathbb{P}:\mathcal{F}\to[0,1]$ is a probability measure.  With $\mathcal{B}^{N}:=\mathcal{B}(\mathbb{A}^{N})$ representing the Borel $\sigma$-algebra on $\mathbb{A}^N \subseteq \mathbb{R}^N$, consider an $\mathbb{A}^{N}$-valued input random vector $\mathbf{X}:=(X_{1},\ldots,X_{N})^T:(\Omega,\mathcal{F})\to(\mathbb{A}^{N},\mathcal{B}^{N})$, describing the statistical uncertainties in all system parameters of a stochastic problem.  The input random variables are generally dependent and are also referred to as basic random variables.  The integer $N$ represents the number of input random variables and is often referred to as the dimension of the stochastic problem.

Denote by $F_{\mathbf{X}}({\mathbf{x}}):=\mathbb{P}(\cap_{i=1}^{N}\{ X_i \le x_i \})$ the joint distribution function of $\mathbf{X}$, admitting the joint probability density function $f_{\mathbf{X}}({\mathbf{x}}):={\partial^N F_{\mathbf{X}}({\mathbf{x}})}/{\partial x_1 \cdots \partial x_N}$. Given the abstract probability space $(\Omega,\mathcal{F},\mathbb{P})$, the image probability space is $(\mathbb{A}^N,\mathcal{B}^{N},f_{\mathbf{X}}d\mathbf{x})$, where $\mathbb{A}^N$ can be viewed as the image of $\Omega$ from the mapping $\mathbf{X}:\Omega \to \mathbb{A}^N$, and is also the support of $f_{\mathbf{X}}({\mathbf{x}})$.  Relevant statements and objects in one space have obvious counterparts in the other space.  Both probability spaces will be used in this paper.

A set of assumptions used or required by GPDD is as follows.

\begin{assumption}
The random vector $\mathbf{X}:=(X_{1},\ldots,X_{N})^T:(\Omega,\mathcal{F})\to(\mathbb{A}^{N},\mathcal{B}^{N})$
\begin{enumerate}[{$(1)$}]
\item
has an absolutely continuous joint distribution function $F_{\mathbf{X}}({\mathbf{x}})$ and a continuous joint probability density function $f_{\mathbf{X}}({\mathbf{x}})$ with a bounded or unbounded support $\mathbb{A}^N \subseteq \mathbb{R}^N$;
\item
possesses absolute finite moments of all orders, that is, for all $\mathbf{j}:=(j_1,\ldots,j_N) \in \mathbb{N}_0^N$,
\begin{equation}
\mathbb{E} \left[ | \mathbf{X}^{\mathbf{j}} | \right] :=
\int_{\Omega}|\mathbf{X}(\omega)^{\mathbf{j}}|d\mathbb{P}(\omega) =
\int_{\mathbb{A}^N} | \mathbf{x}^{\mathbf{j}} | f_{\mathbf{X}}({\mathbf{x}}) d\mathbf{x} < \infty,
\label{2.1}
\end{equation}
where $\mathbf{X}^{\mathbf{j}}=X_1^{j_1}\cdots X_N^{j_N}$ and $\mathbb{E}$ is the expectation operator with respect to the probability measure $\mathbb{P}$ or $f_{\mathbf{X}}(\mathbf{x}) d\mathbf{x}$;
\item
has a joint probability density function $f_{\mathbf{X}}({\mathbf{x}})$, which
\begin{enumerate}
\item
has a compact support, that is, there exists a compact subset $\mathbb{A}^N \subset \mathbb{R}^N$ such that $\mathbb{P}(\mathbf{X} \in \mathbb{A}^N)=1$, or
\item
is exponentially integrable, that is, there exists a real number $a > 0$ such that
\begin{equation}
\int_{\mathbb{A}^N}
\exp{\left( a \| \mathbf{x} \|\right)
f_{\mathbf{X}}(\mathbf{x}) d\mathbf{x}} < \infty,
\label{2.2}
\end{equation}
where $\|\cdot\|:\mathbb{A}^N \to \mathbb{R}_0^+$ is an arbitrary norm; and
\end{enumerate}
\item
has a joint probability density function $f_{\mathbf{X}}({\mathbf{x}})$ with a grid-closed support, that is, there exists a grid for every point $\mathbf{x}$ of $\text{supp}(f_{\mathbf{X}})=\mathbb{A}^N \subseteq \mathbb{R}^N$.
\end{enumerate}
\label{a1}
\end{assumption}

Item (1) of Assumption \ref{a1} is not essential, but it is commonly used in stochastic applications.  Item (2) of Assumption \ref{a1} confirms the existence of an infinite sequence of multivariate orthogonal polynomials consistent with the input probability measure.  Item (3) of Assumption \ref{a1}, in addition to Items (1) and (2), guarantees the input probability measure to be determinate\footnote{The density function of the probability measure, if it is uniquely determined by a sequence of moments, is called determinate or M-determinate.  Otherwise, the density function is indeterminate or M-indeterminate.  This is known as the moment problem with three prominent flavors, depending on the support of the density: Hausdorff moment problem ($\mathbb{A}^N = [0,1]^N$), Stieltjes moment problem ($\mathbb{A}^N = \mathbb{R}_0^{+N}$), and Hamburger moment problem ($\mathbb{A}^N = \mathbb{R}^N $).},
resulting in a complete orthogonal polynomial system and hence a basis of a function space of interest.  Item (4) of Assumption \ref{a1} ensures that for any point $\mathbf{x} \in \text{supp}(f_{\mathbf{X}})$, one can traverse in each coordinate direction and find another point $\mathbf{x}' \in \text{supp}(f_{\mathbf{X}})$ \cite{hooker07}.  Examples of input random variables satisfying Assumption \ref{a1} are multivariate Gaussian, exponential, Laplace, and Dirichlet variables, provided that the density parameters are chosen appropriately.  This assumption, to be explained in the next section, is vitally important for the determinacy of the probability measure and the completeness of orthogonal polynomials.  Examples where Items (1) and (2) are satisfied, but Item (3) is not, are lognormal distributions, select distributions from the Farlie-Gumbel-Morgenstern family, Kotz-type distributions, and cases involving nonlinear transformations of random variables with determinate distributions \cite{kleiber13}.  An example where Item (4) is violated is when the input probability density function is defined on a line. Compared with others, this imposes only a mild restriction on the probability measure. Unless otherwise specified, Assumption \ref{a1} is used throughout the paper.

\section{A generalized ANOVA dimensional decomposition}
Let $y(\mathbf{X}):=y(X_{1},\ldots,X_{N})\in L^2(\Omega, \mathcal{F}, \mathbb{P})$ be a real-valued, square-integrable output random variable defined on the same probability space $(\Omega, \mathcal{F}, \mathbb{P})$.  The vector space $L^2(\Omega, \mathcal{F}, \mathbb{P})$ is a Hilbert space such that
\[
\mathbb{E}\left[ y^2(\mathbf{X})\right]:=
\int_{\Omega} y^2(\mathbf{X}(\omega))d\mathbb{P}(\omega)=
\int_{\mathbb{A}^N} y^2(\mathbf{x})f_{\mathbf{X}}(\mathbf{x})d\mathbf{x} < \infty
\]
with inner product
\[
\left( y(\mathbf{X}),z(\mathbf{X}) \right)_{L^2(\Omega, \mathcal{F}, \mathbb{P})}  :=
\int_{\Omega} y(\mathbf{X}(\omega))z(\mathbf{X}(\omega))d\mathbb{P}(\omega)=
\int_{\mathbb{A}^N} y(\mathbf{x})z(\mathbf{x})f_{\mathbf{X}}(\mathbf{x})d\mathbf{x} =:
\left( y(\mathbf{x}),z(\mathbf{x}) \right)_{f_{\mathbf{X}}d\mathbf{x}}
\]
and norm
\[
\|y(\mathbf{X})\|_{L^2(\Omega, \mathcal{F}, \mathbb{P})}:=
\sqrt{(y(\mathbf{X}),y(\mathbf{X}))_{L^2(\Omega, \mathcal{F}, \mathbb{P})}}=
\sqrt{(y(\mathbf{x}),y(\mathbf{x}))_{f_{\mathbf{X}}d\mathbf{x}}}=:
\|y(\mathbf{x})\|_{f_{\mathbf{X}}d\mathbf{x}}.
\]
It is elementary to show that $y(\mathbf{X}(\omega)) \in L^2(\Omega,\mathcal{F},\mathbb{P})$ if and only if $y(\mathbf{x}) \in L^2(\mathbb{A}^N,\mathcal{B}^{N},f_{\mathbf{X}}d\mathbf{x})$.

\subsection{Generalized ADD}
For $N\in\mathbb{N}$, denote by $\{1,\ldots,N\}$ an index set, so that $u \subseteq \{1,\ldots,N\}$ is a subset, including the empty set $\emptyset$, with cardinality $0 \le |u| \le N$. The complementary subset of $u$ is denoted by $-u:=\{1,\cdots,N\}\backslash u$.
For $\emptyset \ne u \subseteq \{1,\ldots,N\}$, let $\mathbf{X}_u:=(X_{i_1},\ldots,X_{i_{|u|}})^T$, $1\leq i_{1}<\cdots<i_{|u|}\leq N$, a subvector of $\mathbf{X}$, be defined on the abstract probability space $(\Omega^u,\mathcal{F}^u,\mathbb{P}^u)$, where $\Omega^u$ is the sample space of $\mathbf{X}_u$, $\mathcal{F}^u$ is a $\sigma$-algebra on $\Omega^u$, and $\mathbb{P}^u$ is a probability measure. The complementary subvector is defined by $\mathbf{X}_{-u}:=\mathbf{X}_{\{1,\cdots,N\}\backslash u}$.  The corresponding image probability space is $(\mathbb{A}^u,\mathcal{B}^{u},f_{\mathbf{X}_u}d\mathbf{x}_u)$, where $\mathbb{A}^u \subseteq \mathbb{R}^{|u|}$ is the image sample space of $\mathbf{X}_u$, $\mathcal{B}^{u}$ is the Borel $\sigma$-algebra on $\mathbb{A}^u$, and $f_{\mathbf{X}_u}(\mathbf{x}_u):=\int_{\mathbb{A}^{-u}}f_{\mathbf{X}}(\mathbf{x})d\mathbf{x}_{-u}$ is the marginal probability density function of $\mathbf{X}_u$ supported on $\mathbb{A}^u$.

Under Item (4) of Assumption \ref{a1}, fulfilled by common probability
distributions, a square-integrable function $y(\mathbf{X}) \in L^2(\Omega,\mathcal{F},\mathbb{P})$ of input variables $\mathbf{X}$ admits a unique, finite, hierarchical\footnote{The adjective \emph{hierarchical} is used here in the context of dimension-wise hierarchy of input variables.} expansion \cite{hooker07}
\begin{equation}
y(\mathbf{X})={\displaystyle \sum_{u\subseteq\{1,\cdots,N\}}y_{u}(\mathbf{X}_{u})},
\label{3.1}
\end{equation}
referred to as the generalized ADD \cite{rahman14b}, in terms of the component functions
$y_{u}$, $u\subseteq\{1,\cdots,N\}$, of input variables with increasing dimensions. Here, $y_{u}(\mathbf{X}_u)=y_{u}(X_{i_1},\ldots,X_{i_{|u|}})$ is a $|u|$-variate component function of $y$, describing a constant or a $|u|$-variate interaction of $\mathbf{X}_u:=(X_{i_1},\ldots,X_{i_{|u|}})$ on $y$ when $|u|=0$ or $|u|>0$.
Similar to the classical ADD \cite{griebel10,hoeffding48, kuo10, rabitz99, rahman14,sobol93}, the summation in \eqref{3.1} comprises $2^{N}$ component functions, with each function depending on a group of variables indexed by a particular subset of
$\{1,\cdots,N\}$, including the empty set $\emptyset$.

\subsection{Component functions of generalized ADD}
A simple way to link the component functions of ADD, be it classical or generalized, to the function $y$ involves exploiting annihilating conditions.  However, the original annihilating conditions \cite{rabitz99,rahman14,sobol93}, applicable to the classical ADD, are too strong for dependent random variables and hence not appropriate for the generalized ADD.  Therefore, the conditions must be weakened to the degree possible under Item (4) of Assumption \ref{a1}.  Indeed, there exist such weak
annihilating conditions, which mandate all non-constant component functions $y_{u}$ of the generalized ADD to integrate to \emph{zero} with respect to the marginal density function $f_{\mathbf{X}_u}(\mathbf{x}_u)$ of $\mathbf{X}_u$ in each coordinate direction of $u$, that is \cite{hooker07},
\begin{equation}
\int_{\mathbb{A}^{\{i\}}}y_{u}(\mathbf{x}_{u})f_{\mathbf{X}_u}(\mathbf{x}_u)dx_{i}=0\;\mathrm{for}\; i\in u\neq\emptyset.
\label{3.2}
\end{equation}
Compared with the original annihilating conditions, \eqref{3.2} represents a milder version, but it still produces two remarkable properties of the generalized ADD \cite{rahman14b}:

\vspace{0.1in}
\begin{enumerate}[{$(1)$}]
\item
The generalized ADD component functions $y_{u}$,
where $\emptyset\ne u\subseteq\{1,\cdots,N\}$, have zero means, that is,
\begin{equation}
\mathbb{E}\left[y_{u}(\mathbf{X}_{u})\right]=0.
\label{3.3}
\end{equation}
\item
Two distinct generalized ADD component functions $y_{u,G}$
and $y_{v,G}$, where $\emptyset\ne u\subseteq\{1,\cdots,N\}$, $\emptyset\ne v\subseteq\{1,\cdots,N\}$,
and $v\subset u$, are orthogonal, that is, they satisfy the property
\begin{equation}
\mathbb{E}\left[y_{u}(\mathbf{X}_{u})y_{v}(\mathbf{X}_{v})\right]=0.
\label{3.4}
\end{equation}
\end{enumerate}

Applying the weak annihilating conditions \eqref{3.2} and the second-moment properties from \eqref{3.3} and \eqref{3.4}, the master formulae for all component functions $y_{u}$, $u\subseteq\{1,\cdots,N\}$, of the generalized ADD are \cite{rahman14b}
\begin{subequations}
\begin{align}
y_{\emptyset} & =\int_{\mathbb{A}^{N}}y(\mathbf{x})f_{\mathbf{X}}(\mathbf{x})d\mathbf{x},
\label{3.5a}\\
y_{u}(\mathbf{X}_{u}) & =\int_{\mathbb{A}^{-u}}y(\mathbf{X}_{u},\mathbf{x}_{-u})f_{\mathbf{X}_{-u}}(\mathbf{x}_{-u})d\mathbf{x}_{-u}-{\displaystyle \sum_{v\subset u}y_{v}(\mathbf{X}_{v})}-\nonumber \\
 & \;\;\;{\displaystyle \sum_{{\textstyle {\emptyset\ne v\subseteq\{1,\cdots,N\}\atop v\cap u\ne\emptyset,v\nsubseteq u}}}{\displaystyle \int_{\mathbb{A}^{v\cap-u}}y_{v}(\mathbf{X}_{v \cap u},\mathbf{x}_{v \cap -u})
 f_{\mathbf{X}_{v\cap-u}}(\mathbf{x}_{v\cap-u})d\mathbf{x}_{v\cap-u}}}.
 \label{3.5b}
\end{align}
\end{subequations}
Although the original formulae were derived using $\mathbb{A}^N=\mathbb{R}^N$ \cite{rahman14b}, the extension for the case of $\mathbb{A}^N \subseteq \mathbb{R}^N$ is trivial.  Here, $(\mathbf{X}_{u},\mathbf{x}_{-u})$ denotes an $N$-dimensional vector whose $i$th component is $X_{i}$ if $i\in u$ and $x_{i}$ if $i\notin u$.  When $u=\emptyset$, both sums in \eqref{3.5b} vanish, resulting in the expression of the constant function $y_{\emptyset}$ in \eqref{3.5a}. When $u=\{1,\cdots,N\}$, the integration in the first line of \eqref{3.5b} is on the empty set and the sum in the second line of \eqref{3.5b} vanishes, reproducing \eqref{3.1} and hence finding the last function $y_{\{1,\cdots,N\}}$. Indeed, all component functions of $y$ can be obtained by interpreting literally \eqref{3.5b}.

From \eqref{3.5a} and \eqref{3.5b}, two important observations stand out.  First, the constant component function of ADD, whether classical or generalized, is the same as the expected value of $y(\mathbf{X})$. Second, in contrast to the classical ADD, the component functions of the generalized ADD, satisfying \eqref{3.5b}, are coupled and must be solved simultaneously. In this case, for a given $\emptyset\neq u\subseteq\{1,\cdots,N\}$, the component function $y_{u}$ depends not only on the component functions $y_{v}$, where $v\subset u$, but also on the component functions $y_{v}$, where $v\cap u\ne\emptyset$, $v\nsubseteq u$.  Readers interested in further details of the generalized ADD are directed to a prior work of the author \cite{rahman14b}.

The generalized ADD discussed in the preceding paragraph can also be obtained by splitting the Hilbert space
\begin{equation}
L^2(\mathbb{A}^N,\mathcal{B}^{N},f_{\mathbf{X}}d\mathbf{x}) =
\displaystyle \boldsymbol{1} \oplus  \bigcup_{\emptyset \ne u \subseteq \{1,\ldots,N\}} \mathcal{W}_u,
\label{3.6}
\end{equation}
where
\begin{equation}
\mathcal{W}_u :=
\left\{y_u \in L^2(\mathbb{A}^u,\mathcal{B}^{u},f_{\mathbf{X}_u}d\mathbf{x}_u):
\int_{\mathbb{A}^{\{i\}}}y_{u}(\mathbf{x}_{u})f_{\mathbf{X}_u}(\mathbf{x}_u)dx_{i}=0\;\mathrm{for}\; i\in u\neq\emptyset
\right\}
\label{3.7}
\end{equation}
is a generalized ADD subspace comprising $|u|$-variate component functions of $y$.  However, the subspaces $\mathcal{W}_u$, $\emptyset \ne u\subseteq\{1,\ldots,N\}$, are in general infinite-dimensional; therefore, further discretization of $\mathcal{W}_u$ is necessary.  For instance, by introducing the measure-consistent multivariate orthogonal polynomial basis to be introduced in Section 4, a component function $y_u(\mathbf{X}_u) \in \mathcal{W}_u$ can be expressed as a linear combination of these basis functions.  The result is a polynomial refinement of the generalized ADD, namely GPDD, which is the principal motivation for this work.

\section{Multivariate orthogonal polynomials}
Let $\mathbf{X}$ be an input random vector with a general probability measure $f_{\mathbf{X}}(\mathbf{x})d\mathbf{x}$ on $\mathbb{A}^N$, satisfying Items (1)-(3) of Assumption \ref{a1}. It is elementary to show that the same assumptions are also valid for any subvector $\mathbf{X}_u:=(X_{i_1},\ldots,X_{i_{|u|}})^T$ with the marginal probability measure $f_{\mathbf{X}_u}(\mathbf{x}_u)d\mathbf{x}_u$ on $\mathbb{A}^u$. When $\emptyset \ne u \subseteq \{1,\ldots,N\}$, a $|u|$-dimensional multi-index is denoted by $\mathbf{j}_u:=(j_{i_1},\ldots,j_{i_{|u|}})\in \mathbb{N}_0^{|u|}$ with total degree $|\mathbf{j}_u|:=j_{i_1}+ \cdots +j_{i_{|u|}}$, where $j_{i_p} \in \mathbb{N}_0$, $p=1,\ldots,|u|$, represents the $p$th component of $\mathbf{j}_u$.\footnote{The same symbol $|\cdot|$ is used for designating both the cardinality of a set and the degree of a multi-index in this paper.}

\subsection{Measure-consistent orthogonal polynomials}
Denote by
\[
\Pi^{u}:=\mathbb{R}[\mathbf{x}_u]=\mathbb{R}[x_{i_1},\ldots,x_{i_{|u|}}]
\]
the space of all real polynomials in $\mathbf{x}_u$.  For any polynomial pair $P_u,Q_u \in \Pi^u$, $\emptyset \ne u\subseteq\{1,\ldots,N\}$, define an inner product
\begin{equation}
{\left(P_u,Q_u \right)}_{f_{\mathbf{X}_u}d\mathbf{x}_u}:=
\int_{\mathbb{A}^{u}}P_u(\mathbf{x}_u)Q_u(\mathbf{x}_u)
f_{\mathbf{X}_u}(\mathbf{x}_u)d\mathbf{x}_u=
\mathbb{E}\left[ P_u(\mathbf{X}_u)Q_u(\mathbf{X}_u) \right]
\label{4.1}
\end{equation}
on $\Pi^u$ with respect to the measure $f_{\mathbf{X}_u}(\mathbf{x}_u)d\mathbf{x}_u$ and the induced norm
\[
\|P_u\|_{f_{\mathbf{X}_u}d\mathbf{x}_u}:=
\sqrt{{(P_u,P_u)}_{f_{\mathbf{X}_u}d\mathbf{x}_u}}=
\left(\int_{\mathbb{A}^{u}}P_u^2(\mathbf{x}_u)f_{\mathbf{X}_u}(\mathbf{x}_u)d\mathbf{x}_u \right)^{1/2}=
\sqrt{\mathbb{E}\left[ P_u^2(\mathbf{X}_u) \right]}.
\]
The polynomials $P_u \in \Pi^u$ and $Q_u \in \Pi^u$ are called orthogonal to each other with respect to $f_{\mathbf{X}_u}(\mathbf{x}_u)d\mathbf{x}_u$ if ${(P_u,Q_u)}_{f_{\mathbf{X}_u}d\mathbf{x}_u}=0$. Moreover, a polynomial $P_u \in \Pi^u$ is said to be an orthogonal polynomial with respect to $f_{\mathbf{X}_u}(\mathbf{x}_u)d\mathbf{x}_u$ if it is orthogonal to all polynomials of lower degree, that is, if \cite{dunkl14}
\begin{equation}
{\left(P_u,Q_u \right)}_{f_{\mathbf{X}_u}d\mathbf{x}_u}=0 ~\forall Q_u \in \Pi^u ~\text{with}~
\deg Q_u < \deg P_u.
\label{4.2}
\end{equation}
Under Items (1) and (2) of Assumption \ref{a1}, moments of $\mathbf{X}_u$ of all orders exist and are finite, so that the inner product in \eqref{4.1} is well defined.  As the inner product is positive-definite, clearly $\|P_u\|_{f_{\mathbf{X}_u}d\mathbf{x}_u}> 0$ for all non-zero $P_u \in \Pi^u$.  Then, there exists an infinite set of multivariate orthogonal polynomials, say,
$\{ P_{u,\mathbf{j}_u}(\mathbf{x}_u):\mathbf{j}_u \in \mathbb{N}_0^{|u|} \}$, $P_{u,\boldsymbol{0}}=1$, $P_{u,\mathbf{j}_u} \ne 0$, which is consistent with the probability measure $f_{\mathbf{X}_u}(\mathbf{x}_u)d\mathbf{x}_u$, satisfying
\begin{equation}
{\left(P_{u,\mathbf{j}_u},P_{u,\mathbf{k}_u} \right)}_{f_{\mathbf{X}_u}d\mathbf{x}_u} = 0
~\text{whenever}~|\mathbf{j}_u| \ne |\mathbf{k}_u|
\label{4.3}
\end{equation}
for $\mathbf{k}_u \in \mathbb{N}_0^{|u|}$. Here, the multi-index $\mathbf{j}_u$ of the multivariate polynomial $P_{u,\mathbf{j}_u}(\mathbf{x}_u)$ refers to its total degree $|\mathbf{j}_u|$.  Clearly, each $P_{u,\mathbf{j}_u} \in \Pi^u$ is an orthogonal polynomial satisfying \eqref{4.2}. This means that $P_{u,\mathbf{j}_u}$ is orthogonal to all polynomials of different degrees, but it may not be orthogonal to other orthogonal polynomials of the same degree.

Consider for each $l \in \mathbb{N}_0$ the elements of the set $\{ \mathbf{j}_u \in \mathbb{N}_0^{|u|}: |\mathbf{j}_u|=l \}$, $l \in \mathbb{N}_0$, which is arranged as $\mathbf{j}_u^{(1)},\ldots,\mathbf{j}_u^{(K_{u,l})}$ according to a monomial order of choice.  The set has cardinality
\[
K_{u,l}=\#\left\{\mathbf{j}_u \in \mathbb{N}_0^{|u|}: |\mathbf{j}_u|=l \right\} =
\binom{|u|+l-1}{l}.
\]
Denote by
\[
\mathbf{x}_{u,l} = (\mathbf{x}_u^{\mathbf{j}_u^{(1)}},\ldots,\mathbf{x}_u^{\mathbf{j}_u^{(K_{u,l})}})^T
\]
the $K_{u,l}$-dimensional column vector whose elements are the monomials
$\mathbf{x}_u^{\mathbf{j}_u}$ for $|\mathbf{j}_u|=l$ and by
\[
\mathbf{P}_{u,l}(\mathbf{x}_u):=
(P_{u,\mathbf{j}_u^{(1)}}(\mathbf{x}_u),\ldots,P_{u,\mathbf{j}_u^{(K_{u,l})}}(\mathbf{x}_u))^T
\]
the $K_{u,l}$-dimensional column vector whose elements are obtained from the polynomial sequence
$\{P_{u,\mathbf{j}_u}(\mathbf{x}_u)\}_{|\mathbf{j}_u|=l}$, both arranged in the aforementioned order. This leads to a formal definition of multivariate orthogonal polynomials.

\begin{definition}[Dunkl and Xu \cite{dunkl14}]
Let ${\left(\cdot,\cdot \right)}_{f_{\mathbf{X}_u}d\mathbf{x}_u}:\Pi^u \times \Pi^u \to \mathbb{R}$ be an inner product. A set of polynomials $\{ P_{u,\mathbf{j}_u}(\mathbf{x}_u): |\mathbf{j}_u|=l, \mathbf{j}_u \in \mathbb{N}_0^{|u|}\}$, $P_{u,\mathbf{j}_u}(\mathbf{x}_u) \in \Pi_l^{|u|}$, of degree $l$ or its $K_{u,l}$-dimensional
column vector $\mathbf{P}_{u,l}(\mathbf{x}_u)$, is said to be orthogonal with respect to the inner product $(\cdot,\cdot)_{f_{\mathbf{X}_u}d\mathbf{x}_u}$, or alternatively with respect to the probability measure $f_{\mathbf{X}_u}(\mathbf{x}_u)d\mathbf{x}_u$, if, for $l,r \in \mathbb{N}_0$,
\begin{equation}
{\left(\mathbf{x}_{u,r},\mathbf{P}_{u,l}^T(\mathbf{x}_u) \right)}_{f_{\mathbf{X}_u}d\mathbf{x}_u}:=
\int_{\mathbb{A}^u}\mathbf{x}_{u,r}\mathbf{P}_{u,l}^T(\mathbf{x}_u)f_{\mathbf{X}_u}(\mathbf{x}_u)d\mathbf{x}_u=:
\mathbb{E} \left[ \mathbf{X}_{u,r}\mathbf{P}_{u,l}^T(\mathbf{X}_u) \right]
=\boldsymbol{0},~~l > r,
\label{4.4}
\end{equation}
where
\begin{equation}
\mathbf{S}_{u,l}:=
{\left(\mathbf{x}_{u,l},\mathbf{P}_{u,l}^T(\mathbf{x}_u) \right)}_{f_{\mathbf{X}_u}d\mathbf{x}_u}:=
\int_{\mathbb{A}^u}\mathbf{x}_{u,l}\mathbf{P}_{u,l}^T(\mathbf{x}_u)f_{\mathbf{X}_u}(\mathbf{x}_u)d\mathbf{x}_u=:
\mathbb{E} \left[ \mathbf{X}_{u,l}\mathbf{P}_{u,l}^T(\mathbf{X}_u) \right]
\label{4.5}
\end{equation}
is a $K_{u,l} \times K_{u,l}$ invertible matrix.
\label{d1}
\end{definition}

Using the vector notation, one can write
\[
\mathbf{P}_{u,r}(\mathbf{x}_u)=\mathbf{H}_{u,r,r}\mathbf{x}_{u,r}+\mathbf{H}_{u,r,r-1}\mathbf{x}_{u,r-1}+\cdots+
\mathbf{H}_{u,r,0}\mathbf{x}_{u,0},~~r \in \mathbb{N}_0,
\]
where $\mathbf{H}_{u,r,r-k}$, $k=0,1,\ldots,r$, are various coefficient matrices of size $K_{u,r} \times K_{u,r-k}$.  Then, using the properties in \eqref{4.4} and \eqref{4.5}, the inner product
\[
{(\mathbf{P}_{u,r}(\mathbf{x}_u),\mathbf{P}_{u,l}^T(\mathbf{x}_u))}_{f_{\mathbf{X}_u}d\mathbf{x}_u}=
\begin{cases}
\boldsymbol{0},                         & l > r, \\
\mathbf{H}_{u,l,l} \mathbf{S}_{u,l}     & l = r.
\end{cases}
\]
Therefore, Definition \ref{d1} agrees with the usual notion of orthogonal polynomials satisfying \eqref{4.2}.  Perhaps the most prominent example of classical multivariate orthogonal polynomials is the case of multivariate Hermite polynomials, which are consistent with the measure defined by a Gaussian density \cite{erdelyi53,holmquist96}.  Readers interested to learn more about orthogonal polynomials in multiple variables with respect to other measures are referred to the works of Appell and de F\'{e}riet \cite{appell26}, Erdelyi \cite{erdelyi53}, Krall and Sheffer \cite{krall67}, and Dunkl and Xu \cite{dunkl14}.

For general probability measures, established numerical techniques, such as the Gram-Schmidt orthogonalization process \cite{golub96}, can be applied to a sequence of monomials $\{ \mathbf{x}_u^{\mathbf{j}_u}\}_{\mathbf{j}_u\in \mathbb{N}_0^{|u|}}$ with respect to the inner product in \eqref{4.1} to generate a corresponding sequence of any measure-consistent orthogonal polynomials.  However, an important difference between univariate polynomials and multivariate polynomials is the lack of an obvious natural order in the latter. The natural order for monomials of univariate polynomials is the degree order; that is, one orders monomials according to their degree.  For multivariate polynomials, there are many options, such as lexicographic order, graded lexicographic order, and reversed graded lexicographic order, to name just three.  There is no natural choice, and different orders will give different sequences of orthogonal polynomials from the Gram-Schmidt process. It is important to emphasize that the space of multivariate orthogonal polynomials for a generally non-product-type density function cannot be constructed by the tensor product of the spaces of univariate orthogonal polynomials.

Once the multivariate orthogonal polynomials are obtained, they can be scaled to generate their standardized version, as follows.

\begin{definition}
A standardized multivariate orthogonal polynomial $\Psi_{u,\mathbf{j}_u}(\mathbf{x}_u$), $\emptyset \ne u \subseteq \{1,\ldots,N\}$, $\mathbf{j}_u \in \mathbb{N}_0^{|u|}$, of degree $|\mathbf{j}_u|=j_{i_1}+\cdots+j_{i_{|u|}}$ that is consistent with the probability measure $f_{\mathbf{X}_u}(\mathbf{x}_u)d\mathbf{x}_u$ is defined as
\begin{equation}
\Psi_{u,\mathbf{j}_u}(\mathbf{x}_u) :=
\frac{P_{u,\mathbf{j}_u}(\mathbf{x}_u)}
{\|P_{u,\mathbf{j}_u}\|_{f_{\mathbf{X}_u}d\mathbf{x}_u}} =
\frac{P_{u,\mathbf{j}_u}(\mathbf{x}_u)}
{\sqrt{\mathbb{E}[ P_{u,\mathbf{j}_u}^2(\mathbf{X}_u)]}}.
\label{4.5b}
\end{equation}
\label{d2}
\end{definition}
The standardization is not absolutely required, but it results in a relatively simpler expression of GPDD, to be presented in Section 5.

\subsection{Dimension-wise decomposition of polynomial spaces}
A decomposition of polynomial spaces entailing dimension-wise splitting leads to GPDD.  Here, to facilitate such splitting of the polynomial space $\Pi^u$ for any $\emptyset \ne u \subseteq \{1,\ldots,N\}$, limit the component $j_{i_p}$ associated with the $i_p$th variable, where $i_p \in u \subseteq \{1,\ldots,N\}$, $p=1,\ldots,|u|$, and $|u|>0$, to take on only positive integer values. In consequence, $\mathbf{j}_u:=(j_{i_1},\ldots,j_{i_{|u|}})\in \mathbb{N}^{|u|}$, the multi-index of $P_{u,\mathbf{j}_u}(\mathbf{x}_u)$, has degree $|\mathbf{j}_u|=j_{i_1}+\cdots+j_{i{|u|}}$, varying from $|u|$ to $\infty$ as $j_{i_1} \ne \cdots j_{i_{|u|}} \ne 0$.

For $\mathbf{j}_u \in \mathbb{N}^{|u|}$ and $\mathbf{x}_u :=(x_{i_1},\ldots,x_{i_{|u}})$, a monomial in the variables $x_{i_1},\ldots,x_{i_{|u|}}$ is the product $\mathbf{x}_u^{\mathbf{j}_u}=x_{i_1}^{j_{i_1}}\ldots x_{i_{|u|}}^{j_{i_{|u|}}}$ and has a total degree $|\mathbf{j}_u|$.  A linear combination of $\mathbf{x}_u^{\mathbf{j}_u}$, where $|\mathbf{j}_u|= l$, $|u| \le l < \infty$, is a homogeneous polynomial in $\mathbf{x}_u$ of degree $l$.  For $\emptyset \ne u \subseteq \{1,\ldots,N\}$, denote by
\[
\mathcal{Q}_{l}^{u}:=\text{span}\{\mathbf{x}_u^{\mathbf{j}_u}:|\mathbf{j}_u|=l,\,
\mathbf{j}_u\in\mathbb{N}^{|u|}\}, ~|u| \le l < \infty,
\]
the space of homogeneous polynomials in $\mathbf{x}_u$ of degree $l$ where the individual degree of each variable is non-zero and by
\[
\Theta_{m}^{u}:=\text{span}\{\mathbf{x}_u^{\mathbf{j}_u}:|u| \le |\mathbf{j}_u| \le m,\,
\mathbf{j}_u\in\mathbb{N}^{|u|}\}, ~|u| \le m < \infty,
\]
the space of polynomials in $\mathbf{x}_u$ of degree at least $|u|$ and at most $m$ where the individual degree of each variable is non-zero.  The dimensions of the vector spaces $\mathcal{Q}_{l}^{u}$ and $\Theta_{m}^{u}$, respectively, are
\[
\dim\mathcal{Q}_{l}^{u}=\#\left\{\mathbf{j}_u \in \mathbb{N}^{|u|}: |\mathbf{j}_u|=l \right\} =
\displaystyle \binom{l-1}{|u|-1}
\]
and
\[
\dim\Theta_{m}^{u} = \sum_{l=|u|}^m \dim\mathcal{Q}_{l}^{u} =
\displaystyle \sum_{l=|u|}^m \displaystyle \binom{l-1}{|u|-1} =
\displaystyle \binom{m}{|u|}.
\]

Let $\mathcal{Z}_{|u|}^{u}:= \Theta_{|u|}^{u}$.  For each $|u|+1 \le l < \infty$, denote by $\mathcal{Z}_l^u \subset \Theta_l^u$ the space of orthogonal polynomials of degree exactly $l$ that are orthogonal to all polynomials in $\Theta_{l-1}^u$, that is,
\[
\mathcal{Z}_{l}^{u}:=
\{P_u \in\Theta_{l}^{u}:{(P_u,Q_u})_{f_{\mathbf{X}_u}d\mathbf{x}_u}=0~\forall \:Q_u\in\Theta_{l-1}^{u}\},
~|u|+1 \le l < \infty.
\]
Then $\mathcal{Z}_l^u$, provided that the support of $f_{\mathbf{X}_u}(\mathbf{x}_u)$ has non-empty interior, is a vector space of dimension
\[
M_{u,l}:=\dim\mathcal{Z}_{l}^{u}=\dim\mathcal{Q}_{l}^{u}=
\displaystyle{\binom{l-1}{|u|-1}}.
\]
Many choices exist for the basis of $\mathcal{Z}_{l}^{u}$. Here, to be formally proved in Section 4.3, select $\{P_{u,\mathbf{j}_u}(\mathbf{x}_u): |\mathbf{j}_u|=l, \mathbf{j}_u \in \mathbb{N}^{|u|} \} \subset \mathcal{Z}_l^u$ to be a basis of $\mathcal{Z}_{l}^{u}$, comprising $M_{u,l}$ number of basis functions. Each basis function $P_{u,\mathbf{j}_u}(\mathbf{x}_u)$ is a multivariate orthogonal polynomial of degree $|\mathbf{j}_u|$ as defined earlier.  Clearly,
\[
\mathcal{Z}_l^u=\text{span}\{ P_{u,\mathbf{j}_u}(\mathbf{x}_u): |\mathbf{j}_u|=l, \mathbf{j}_u \in \mathbb{N}^{|u|} \},~|u| \le l < \infty.
\]

According to \eqref{4.3}, $P_{u,\mathbf{j}_u}(\mathbf{X}_u)$ is orthogonal to $P_{u,\mathbf{k}_u}(\mathbf{X}_u)$ whenever $|\mathbf{j}_u| \ne |\mathbf{k}_u|$. Therefore, any two distinct polynomial subspaces $\mathcal{Z}_{l}^{u}$ and $\mathcal{Z}_{l'}^{u}$, where $\emptyset \ne u \subseteq \{1,\ldots,N\}$, $|u| \le l < \infty$, and $|u| \le l' < \infty$, are orthogonal whenever $l \ne l'$. In consequence, there exists an orthogonal decomposition of
\[
\begin{array}{rcl}
\Theta_m^u & = & \displaystyle \bigoplus_{l=|u|}^m \mathcal{Z}_{l}^{u}=\bigoplus_{l=|u|}^m \text{span}\{ P_{u,\mathbf{j}_u}(\mathbf{x}_u): |\mathbf{j}_u|=l, \mathbf{j}_u \in \mathbb{N}^{|u|} \} \\
           & = & \text{span}\{ P_{u,\mathbf{j}_u}(\mathbf{x}_u): |u| \le |\mathbf{j}_u|\le m, \mathbf{j}_u \in \mathbb{N}^{|u|} \},
\end{array}
\]
where the symbol $\oplus$ represents the orthogonal sum of vector spaces. Moreover, this facilitates a dimension-wise splitting of
\begin{equation}
\begin{array}{rcl}
\Pi^u  &  =  &  \displaystyle \boldsymbol{1} \oplus  \bigcup_{\emptyset \ne v \subseteq u}~ \bigoplus_{l=|v|}^\infty \mathcal{Z}_{l}^{v} = \displaystyle \boldsymbol{1} \oplus  \bigcup_{\emptyset \ne v \subseteq u}~\bigoplus_{l=|v|}^\infty \text{span}\{ P_{v,\mathbf{j}_v}(\mathbf{x}_v): |\mathbf{j}_v|=l, \mathbf{j}_v \in \mathbb{N}^{|v|} \} \\
       &  =  &  \displaystyle \boldsymbol{1} \oplus  \bigcup_{\emptyset \ne v \subseteq u}~ \text{span}\{ P_{v,\mathbf{j}_v}(\mathbf{x}_v): \mathbf{j}_v \in \mathbb{N}^{|v|} \},
\end{array}
\label{4.6}
\end{equation}
where $\boldsymbol{1}:=\text{span}\{1\}$, the constant subspace, needs to be added because the subspace $\mathcal{Z}_{l}^{v}$ excludes constant functions.

Recall that $\Pi^N$ is the space of all real polynomials in $\mathbf{x}$. Then, setting $u=\{1,\ldots,N\}$ in \eqref{4.6} first and then swapping $v$ for $u$ yields yet another dimension-wise splitting of
\begin{equation}
\begin{array}{rcl}
\Pi^N  & = &
\displaystyle \boldsymbol{1} \oplus  \bigcup_{\emptyset \ne u \subseteq \{1,\ldots,N\}}~ \bigoplus_{l=|u|}^\infty \mathcal{Z}_{l}^{u}  \\
       & = &
\displaystyle \boldsymbol{1} \oplus  \bigcup_{\emptyset \ne u \subseteq \{1,\ldots,N\}} \bigoplus_{l=|u|}^\infty
\text{span}\{ P_{u,\mathbf{j}_u}(\mathbf{x}_u): |\mathbf{j}_u|=l, \mathbf{j}_u \in \mathbb{N}^{|u|} \} \\
       & = &
\displaystyle \boldsymbol{1} \oplus  \bigcup_{\emptyset \ne u \subseteq \{1,\ldots,N\}}
\text{span}\{ P_{u,\mathbf{j}_u}(\mathbf{x}_u): \mathbf{j}_u \in \mathbb{N}^{|u|} \}.
\end{array}
\label{4.7}
\end{equation}
Given the dimension-wise splitting of $\Pi^N$, any square-integrable function of input random vector $\mathbf{X}$ can be expanded as a Fourier-like series of hierarchically ordered multivariate orthogonal polynomials in $\mathbf{X}_u$, $\emptyset \ne u \subseteq \{1,\ldots,N\}$.  The expansion defines GPDD, to be formally presented and analyzed in Section 5.

\subsection{Completeness of orthogonal polynomials and basis}
An important question regarding orthogonal polynomials is whether they are complete and constitute a basis in a function space of interest, such as a Hilbert space.  Let  $L^2(\mathbb{A}^N,\mathcal{B}^{N},f_{\mathbf{X}}d\mathbf{x})$
represent a Hilbert space of square-integrable functions with respect to the probability measure $f_{\mathbf{X}}(\mathbf{x}) d\mathbf{x}$ supported on $\mathbb{A}^{N}$. The following two propositions show that, indeed, orthogonal polynomials span various spaces of interest.

\begin{proposition}
\label{p1}
Let $\mathbf{X}:=(X_{1},\ldots,X_{N})^T:(\Omega,\mathcal{F})\to(\mathbb{A}^{N},\mathcal{B}^{N})$, $N\in\mathbb{N}$, be an $N$-dimensional random vector with multivariate probability density function $f_{\mathbf{X}}(\mathbf{x})$, satisfying Items (1)-(3) of Assumption \ref{a1} and $\mathbf{X}_u:=(X_{i_1},\ldots,X_{i_{|u|}})^T:(\Omega^u,\mathcal{F}^u)\to(\mathbb{A}^u,\mathcal{B}^u)$, $\emptyset \ne u \subseteq \{1,\ldots,N\}$, be a subvector of $\mathbf{X}$. Then $\{P_{u,\mathbf{j}_u}(\mathbf{x}_u): |\mathbf{j}_u|=l, \mathbf{j}_u \in \mathbb{N}^{|u|} \}$, the set of multivariate orthogonal polynomials of degree $l$, $|u| \le l < \infty$, consistent with the probability measure $f_{\mathbf{X}_u}(\mathbf{x}_u)d\mathbf{x}_u$, is a basis of $\mathcal{Z}_{l}^{u}$.
\end{proposition}

\begin{proof}
Under Items (1) and (2) of Assumption \ref{a1}, orthogonal polynomials with respect to the probability measure $f_{\mathbf{X}_u}(\mathbf{x}_u)d\mathbf{x}_u$ exist.  Let $\mathbf{a}_{u,l}^T=(a_{u,l}^{(1)},\ldots,a_{u,l}^{(K_{u,l})})$ be a row vector comprising some
constants $a_{u,l}^{(i)} \in \mathbb{R}$, $i=1,\ldots,K_{u,l}$.  Set $\mathbf{a}_{u,l}^T\mathbf{P}_{u,l}(\mathbf{x}_u)=0$. Multiply both sides of the equality from the right by $\mathbf{x}_{u,l}^T$, integrate with respect to the measure $f_{\mathbf{X}_u}(\mathbf{x}_u)d\mathbf{x}_u$ over $\mathbb{A}^u$, and apply transposition to obtain
\begin{equation}
\mathbf{S}_{u,l} \mathbf{a}_{u,l} = \mathbf{0},
\label{4.8}
\end{equation}
where $\mathbf{S}_{u,l}$, defined in \eqref{4.5}, is a $K_{u,l} \times K_{u,l}$ invertible matrix. Therefore, \eqref{4.8} yields $\mathbf{a}_{u,l}=\mathbf{0}$, proving linear independence of the elements of $\mathbf{P}_{u,l}(\mathbf{x}_u)$ or $\{P_{u,\mathbf{j}_u}(\mathbf{x}_u): |\mathbf{j}_u|=l, \mathbf{j}_u \in \mathbb{N}_0^{|u|}\}$.  Obviously, the elements of the subset $\{P_{u,\mathbf{j}_u}(\mathbf{x}_u): |\mathbf{j}_u|=l, \mathbf{j}_u \in \mathbb{N}^{|u|}\}$, excluding the elements associated with zero components of $j_{i_1},\ldots,j_{i_{|u|}}$, are also linearly independent. Furthermore, the dimension $M_{u,l}$ of $\mathcal{Z}_{l}^{u}$ matches exactly the number of elements of the aforementioned subset.  Therefore, the spanning set $\{P_{u,\mathbf{j}_u}(\mathbf{x}_u): |\mathbf{j}_u|=l, \mathbf{j}_u \in \mathbb{N}^{|u|}\}$ forms a basis of $\mathcal{Z}_{l}^{u}$.
\end{proof}

\begin{proposition}
\label{p2}
Let $\mathbf{X}:=(X_{1},\ldots,X_{N})^T:(\Omega,\mathcal{F})\to(\mathbb{A}^{N},\mathcal{B}^{N})$, $N\in\mathbb{N}$, be an $N$-dimensional random vector with multivariate probability density function $f_{\mathbf{X}}(\mathbf{x})$, satisfying Items (1)-(3) of Assumption \ref{a1} and $\mathbf{X}_u:=(X_{i_1},\ldots,X_{i_{|u|}})^T:(\Omega^u,\mathcal{F}^u)\to(\mathbb{A}^u,\mathcal{B}^u)$, $\emptyset \ne u \subseteq \{1,\ldots,N\}$, be a subvector of $\mathbf{X}$. Consistent with the probability measure $f_{\mathbf{X}_u}(\mathbf{x}_u)d\mathbf{x}_u$, let $\{P_{u,\mathbf{j}_u}(\mathbf{x}_u): |\mathbf{j}_u|=l, \mathbf{j}_u \in \mathbb{N}^N\}$, the set of multivariate orthogonal polynomials of degree $l$, be a basis of $\mathcal{Z}_{l}^{u}$.  Then the set of polynomials from the union-sum collection
\begin{equation}
\displaystyle \boldsymbol{1} \oplus  \bigcup_{\emptyset \ne u \subseteq \{1,\ldots,N\}} \bigoplus_{l=|u|}^\infty
\text{span}\{ P_{u,\mathbf{j}_u}(\mathbf{x}_u): |\mathbf{j}_u|=l, \mathbf{j}_u \in \mathbb{N}^{|u|} \}
\label{4.9}
\end{equation}
is dense in $L^2(\mathbb{A}^N,\mathcal{B}^{N},f_{\mathbf{X}}d\mathbf{x})$.  Moreover,
\begin{equation}
L^2(\mathbb{A}^N,\mathcal{B}^{N},f_{\mathbf{X}}d\mathbf{x}) = \overline{
\displaystyle \boldsymbol{1} \oplus \bigcup_{\emptyset \ne u \subseteq \{1,\ldots,N\}} \bigoplus_{l=|u|}^\infty \mathcal{Z}_{l}^{u}
},
\label{4.10}
\end{equation}
where the overline denotes set closure.
\end{proposition}

\begin{proof}
Under Items (1) and (2) of Assumption \ref{a1}, orthogonal polynomials $P_{u,\mathbf{j}_u}(\mathbf{x}_u)$ with respect to the probability measure $f_{\mathbf{X}_u}(\mathbf{x}_u)d\mathbf{x}_u$ exist.  According to Theorem 3.2.18 of Dunkl and Xu \cite{dunkl14} and related discussion, which exploits Items 3(a) and 3(b) of Assumption \ref{a1}, the polynomial space $\Pi^N$ is dense in the space $L^2(\mathbb{A}^N,\mathcal{B}^{N},f_{\mathbf{X}}d\mathbf{x})$.  Therefore, the set of polynomials from \eqref{4.9}, which is equal to $\Pi^{N}$ as per \eqref{4.7}, is dense in $L^2(\mathbb{A}^N,\mathcal{B}^{N},f_{\mathbf{X}}d\mathbf{x})$. Including the limit points of \eqref{4.9} yields \eqref{4.10}.
\end{proof}

\subsection{Statistical properties of random multivariate polynomials}
When the input random variables $X_1,\ldots,X_N$, instead of real variables $x_1,\ldots,x_N$, are inserted in the argument,  the multivariate polynomials $P_{u,\mathbf{j}_u}(\mathbf{X}_u)$ and $\Psi_{u,\mathbf{j}_u}(\mathbf{X}_u)$, where $\emptyset \ne u \subseteq \{1,\ldots,N\}$ and $\mathbf{j}_u \in \mathbb{N}^{|u|}$, become functions of random input variables.  Therefore, it is important to establish their second-moment properties, to be exploited in the remaining part of this section and Section 5.

\begin{lemma}
\label{l1}
Let $\mathbf{X}:=(X_{1},\ldots,X_{N})^T:(\Omega,\mathcal{F})\to(\mathbb{A}^{N},\mathcal{B}^{N})$, $N\in\mathbb{N}$, be an $N$-dimensional random vector with multivariate probability density function $f_{\mathbf{X}}(\mathbf{x})$, satisfying Items (1)-(4) of Assumption \ref{a1} and $\mathbf{X}_u:=(X_{i_1},\ldots,X_{i_{|u|}})^T:(\Omega^u,\mathcal{F}^u)\to(\mathbb{A}^u,\mathcal{B}^u)$, $\emptyset \ne u \subseteq \{1,\ldots,N\}$, be a subvector of $\mathbf{X}$. Consistent with the probability measure $f_{\mathbf{X}_u}(\mathbf{x}_u)d\mathbf{x}_u$, let $\{P_{u,\mathbf{j}_u}(\mathbf{x}_u): \mathbf{j}_u \in \mathbb{N}^N\}$ be an infinite set of multivariate orthogonal polynomials.  Then each polynomial of the set satisfies the weak annihilating conditions, that is,
\begin{equation}
\int_{\mathbb{A}^{\{i\}}}P_{u,\mathbf{j}_u}(\mathbf{x}_u) f_{\mathbf{X}_u}(\mathbf{x}_u)dx_{i}=0\;\mathrm{for}\; i\in u\ne\emptyset,\;\mathbf{j}_u \in \mathbb{N}^{|u|}.
\label{4.11}
\end{equation}
\end{lemma}

\begin{proof}
Let $y(\mathbf{x}) \in L^2(\mathbb{A}^N,\mathcal{B}^{N},f_{\mathbf{X}}d\mathbf{x})$ be an arbitrary function, where the input probability measure satisfies Items (3) and (4) of Assumption \ref{a1}.  Therefore, as explained in Section 2, there exists a unique generalized ADD of $y(\mathbf{x})$ where the component functions $y_{u}(\mathbf{x}_u)$, $\emptyset \ne u \subseteq \{1,\ldots,N\}$, obey the weak annihilating conditions described in \eqref{3.2}.  Since $y$ is a square-integrable function, so are the component functions of $y$, that is, $y_u(\mathbf{x}_u)$ is an element of $\mathcal{W}_u$ defined in \eqref{3.7}.  Furthermore, from similar considerations given to Propositions \ref{p1} and \ref{p2}, one can prove that the set $\{P_{u,\mathbf{j}_u}(\mathbf{x}_u): \mathbf{j}_u \in \mathbb{N}^N\}$ is a basis of $\mathcal{W}_u$.  Consequently, there exists a Fourier-like series expansion of
\begin{equation}
y_u(\mathbf{x}_u) \sim
\displaystyle
\sum_{\mathbf{j}_u \in \mathbb{N}^{|u|}} \hat{C}_{u,\mathbf{j}_u} P_{u,\mathbf{j}_u}(\mathbf{x}_u)
\label{4.12}
\end{equation}
with $\hat{C}_{u,\mathbf{j}_u}$ denoting the associated expansion coefficients. The coefficients depend on $y_u(\mathbf{x}_u)$, which, in turn, depends on $y$.  Here, the symbol $\sim$ represents equality in a weaker sense, such as equality in mean-square, but not necessarily pointwise nor almost everywhere.  From the standard Hilbert space theory, the infinite series on the right hand side of \eqref{4.12} converges in mean-square.  Combining \eqref{3.2} and \eqref{4.12} and interchanging the integral and summation operators, which is admissible for the convergent sum, yields
\begin{equation}
\displaystyle
\sum_{\mathbf{j}_u \in \mathbb{N}^{|u|}} \hat{C}_{u,\mathbf{j}_u}
\int_{\mathbb{A}^{\{i\}}} P_{u,\mathbf{j}_u}(\mathbf{x}_u) f_{\mathbf{X}_u}(\mathbf{x}_u)dx_{i}=0
\label{4.12b}
\end{equation}
for $i\in u\neq\emptyset$ and $\mathbf{j}_u \in \mathbb{N}^{|u|}$.  Since $y(\mathbf{x}) \in L^2(\mathbb{A}^N,\mathcal{B}^{N},f_{\mathbf{X}}d\mathbf{x})$ is arbitrary, so are the component functions $y_u(\mathbf{x}_u) \in \mathcal{W}_u$ and the resultant coefficients $\hat{C}_{u,\mathbf{j}_u}$, yet the sum in \eqref{4.12b} must vanish.  This is only possible if the integral in \eqref{4.12b} vanishes, resulting in \eqref{4.11}.
\end{proof}

\begin{proposition}
Let $\mathbf{X}:=(X_{1},\ldots,X_{N})^T:(\Omega,\mathcal{F})\to(\mathbb{A}^{N},\mathcal{B}^{N})$, $N\in\mathbb{N}$, be an $N$-dimensional random vector with multivariate probability density function $f_{\mathbf{X}}(\mathbf{x})$, satisfying Items (1)-(4) of Assumption \ref{a1}.
For $\emptyset \ne u,v \subseteq \{1,\ldots,N\}$, $\mathbf{j}_u \in \mathbb{N}^{|u|}$, and $\mathbf{k}_v \in \mathbb{N}^{|v|}$, the first- and second-order moments of multivariate orthogonal polynomials, respectively, are
\begin{equation}
\mathbb{E}\left[ P_{u,\mathbf{j}_u}(\mathbf{X}_u) \right] = 0
\label{4.13}
\end{equation}
and
\begin{equation}
\mathbb{E}\left[ P_{u,\mathbf{j}_u}(\mathbf{X}_u) P_{v,\mathbf{k}_v}(\mathbf{X}_v) \right] =
\begin{cases}
0, & u \subset v \subset u,\forall \; \mathbf{j}_u, \mathbf{k}_v , \\
0, & \forall \;u,v,|\mathbf{j}_u|\neq |\mathbf{k}_v|, \\
\displaystyle \int_{\mathbb{A}^{u}}\!\!\!P_{u,\mathbf{j}_u}^2(\mathbf{x}_u)f_{\mathbf{X}_u}(\mathbf{x}_u)d\mathbf{x}_u, & u=v,\mathbf{j}_u=\mathbf{k}_v, \\
\displaystyle \int_{\mathbb{A}^{u\cup v}}\!\!\!\!\!\!P_{u,\mathbf{j}_u}(\mathbf{x}_u) P_{v,\mathbf{k}_v}(\mathbf{x}_v) f_{\mathbf{X}_{u\cup v}}(\mathbf{x}_{u\cup v})d\mathbf{x}_{u\cup v}, & \text{otherwise}.
\end{cases}
\label{4.14}
\end{equation}
\label{p3}
\end{proposition}

\begin{proof}
In reference to Definition \ref{d1}, set $r=0$ and $1 \le l < \infty$ for any $\emptyset \ne u \subseteq \{1,\ldots,N\}$.  Then $\mathbf{x}_{u,0}=(1)^T=(1)$, so that \eqref{4.4} becomes
\[
{\left(1,\mathbf{P}_{u,l}^T(\mathbf{x}_u) \right)}_{f_{\mathbf{X}_u}d\mathbf{x}_u}:=
\int_{\mathbb{A}^u}\mathbf{P}_{u,l}^T(\mathbf{x}_u)f_{\mathbf{X}_u}(\mathbf{x}_u)d\mathbf{x}_u=
\boldsymbol{0}
\]
for all $1 \le l < \infty$.  Since the vector $\mathbf{P}_{u,l}(\mathbf{X}_u)$ comprises as elements the orthogonal polynomials $P_{u,\mathbf{j}_u}(\mathbf{X}_u)$, $|\mathbf{j}_u|=l$, one obtains \eqref{4.13} for any $\emptyset \ne u \subseteq \{1,\ldots,N\}$ and $\mathbf{j}_u \in \mathbb{N}^{|u|}$.

For the first \emph{zero} result of \eqref{4.14}, consider two subsets $\emptyset\ne u,v\subseteq\{1,\cdots,N\}$, where $v\subset u$, $\mathbf{j}_u \in \mathbb{N}^{|u|}$, and $\mathbf{k}_v \in \mathbb{N}^{|v|}$.  Obviously, $u=v\cup(u\setminus v)$. Let $i\in(u\setminus v)\subseteq u$. Then
\begin{align}
\mathbb{E}\left[P_{u,\mathbf{j}_u}(\mathbf{X}_{u})P_{v,\mathbf{k}_v}(\mathbf{X}_{v})\right] & :=\int_{\mathbb{A}^{N}}P_{u,\mathbf{j}_u}(\mathbf{x}_{u})P_{v,\mathbf{k}_v}(\mathbf{x}_{v})f_{\mathbf{X}}(\mathbf{x})d\mathbf{x}\\
 & =\int_{\mathbb{A}^{u}}P_{u,\mathbf{j}_u}(\mathbf{x}_{u})P_{v,\mathbf{k}_v}(\mathbf{x}_{v})f_{\mathbf{X}_{u}}(\mathbf{x}_{u})d\mathbf{x}_{u}\\
 & =\int_{\mathbb{A}^{v}}P_{v,\mathbf{k}_v}(\mathbf{x}_{v})\int_{\mathbb{A}^{u\setminus v}}P_{u,\mathbf{j}_u}(\mathbf{x}_{u})f_{\mathbf{X}_{u}}(\mathbf{x}_{u})d\mathbf{x}_{u\setminus v}d\mathbf{x}_{v}\\
 & =\int_{\mathbb{A}^{v}}P_{v,\mathbf{k}_v}(\mathbf{x}_{v})\int_{\mathbb{A}^{(u\setminus v)\setminus {\{i\}}}}\int_{\mathbb{A}^{\{i\}}}P_{u,\mathbf{j}_u}(\mathbf{x}_{u})f_{\mathbf{X}_{u}}(\mathbf{x}_{u})dx_{i}\prod_{{\textstyle {j\in(u\setminus v)\atop {j\neq i}}}}dx_{j}d\mathbf{x}_{v}\\
 & =0,
\end{align}
where the equality to \emph{zero} in the last line results from the innermost integral vanishing as per \eqref{4.11} in Lemma \ref{l1}.  Interchanging $u$ and $v$ obtains the complete result.

To derive the second \emph{zero} result of \eqref{4.14}, recognize that the polynomials $P_{u,\mathbf{j}_u}(\mathbf{x}_{u})$ and $P_{v,\mathbf{k}_v}(\mathbf{x}_{v})$ are members of $\mathcal{Z}_{|\mathbf{j}_u|}^u$ and $\mathcal{Z}_{|\mathbf{k}_v|}^v$, respectively.  Therefore, they can both be expanded in terms of orthogonal polynomials $P_{\mathbf{j}}(\mathbf{x})$ in $\mathbf{x}$, for instance,
\[
P_{u,\mathbf{j}_u}(\mathbf{x}_{u}) =
\displaystyle \sum_{|\mathbf{j}|=|\mathbf{j}_u|} C_{u,\mathbf{j}} P_{\mathbf{j}}(\mathbf{x}),~
P_{v,\mathbf{k}_v}(\mathbf{x}_{v}) =
\displaystyle \sum_{|\mathbf{k}|=|\mathbf{k}_v|} C_{v,\mathbf{k}} P_{\mathbf{k}}(\mathbf{x}),
\]
with $C_{u,\mathbf{j}}$ and $C_{v,\mathbf{k}}$ denoting the associated expansion coefficients.  Then, for any $u$, $v$, and $|\mathbf{j}_u| \ne |\mathbf{k}_v|$,
\begin{align}
\mathbb{E}\left[P_{u,\mathbf{j}_u}(\mathbf{X}_{u})P_{v,\mathbf{k}_v}(\mathbf{X}_{v})\right] & :=\int_{\mathbb{A}^{N}}P_{u,\mathbf{j}_u}(\mathbf{x}_{u})P_{v,\mathbf{k}_v}(\mathbf{x}_{v})f_{\mathbf{X}}(\mathbf{x})d\mathbf{x}\\
  & =  \int_{\mathbb{A}^{N}}
 \displaystyle \sum_{|\mathbf{j}|=|\mathbf{j}_u|} \displaystyle \sum_{|\mathbf{k}|=|\mathbf{k}_v|}
 C_{u,\mathbf{j}} C_{v,\mathbf{k}}
 P_{\mathbf{j}}(\mathbf{x})P_{\mathbf{k}}(\mathbf{x})f_{\mathbf{X}}(\mathbf{x})d\mathbf{x}\\
 & =\displaystyle \sum_{|\mathbf{j}|=|\mathbf{j}_u|} \displaystyle \sum_{|\mathbf{k}|=|\mathbf{k}_v|}
 C_{u,\mathbf{j}} C_{v,\mathbf{k}}
 \int_{\mathbb{A}^{N}}P_{\mathbf{j}}(\mathbf{x})P_{\mathbf{k}}(\mathbf{x})f_{\mathbf{X}}(\mathbf{x})d\mathbf{x}\\
  & =0,
\end{align}
where the equality to \emph{zero} in the last line stems from the vanishing integral according to \eqref{4.3} with $u=\{1,\ldots,N\}$.

Finally, the non-zero expressions of \eqref{4.14} come from their definitions. No further reduction is possible for a general probability measure.
\end{proof}

\begin{corollary}
For $\emptyset \ne u,v \subseteq \{1,\ldots,N\}$, $\mathbf{j}_u \in \mathbb{N}^{|u|}$, and $\mathbf{k}_v \in \mathbb{N}^{|v|}$, the first- and second-order moments of standardized multivariate orthogonal polynomials, respectively, are
\begin{equation}
\mathbb{E}\left[ \Psi_{u,\mathbf{j}_u}(\mathbf{X}_u) \right] = 0
\end{equation}
\vspace{-0.15in}
and
\begin{equation}
\mathbb{E}\left[ \Psi_{u,\mathbf{j}_u}(\mathbf{X}_u) \Psi_{v,\mathbf{k}_v}(\mathbf{X}_v) \right] =
\begin{cases}
0, & u \subset v \subset u,\forall \; \mathbf{j}_u, \mathbf{k}_v, \\
0, & \forall \; u,v,|\mathbf{j}_u|\neq |\mathbf{k}_v|, \\
1, & u=v,\mathbf{j}_u=\mathbf{k}_v, \\
\displaystyle \int_{\mathbb{A}^{u\cup v}}\!\!\!\!\!\!\Psi_{u,\mathbf{j}_u}(\mathbf{x}_u) \Psi_{v,\mathbf{k}_v}(\mathbf{x}_v) f_{\mathbf{X}_{u\cup v}}(\mathbf{x}_{u\cup v})d\mathbf{x}_{u\cup v}, & \text{otherwise}.
\end{cases}
\end{equation}
\label{c1}
\end{corollary}

\begin{corollary}
Let $\mathbf{X}=(X_1,\ldots,X_N)^T$ be a vector of independent, but not necessarily identical, input random variables, which satisfy Items (1) and (2) Assumption \ref{a1}.  Denote by $f_{X_i}(x_i)$, $i=1,\ldots,N$, the marginal density function of the $i$th random variable $X_i$ and by $\Psi_{\{i\},j_i}(x_i)$ the $j_i$th-degree univariate orthonormal polynomial in $x_i$, which is obtained consistent with the probability measure $f_{X_i}(x_i)dx_i$. Then (1) for $\emptyset \ne u \subseteq \{1,\ldots,N\}$, $\mathbf{j}_u \in \mathbb{N}^{|u|}$,
\[
\Psi_{u,\mathbf{j}_u}(\mathbf{x}_u)=
\prod_{i \in u} \Psi_{\{i\},j_i}(x_i) =
\prod_{p=1}^{|u|} \Psi_{\{i_p\},j_{i_p}}(x_{i_p})
\]
is a multivariate orthonormal polynomial in $\mathbf{x}_u=(x_{i_1},\ldots,x_{i_{|u|}})$ of degree $|\mathbf{j}_u|=j_{i_1}+\cdots+j_{i_{|u|}}$; and (2) for $\emptyset \ne u,v \subseteq \{1,\ldots,N\}$, $\mathbf{j}_u \in \mathbb{N}^{|u|}$, and $\mathbf{k}_v \in \mathbb{N}^{|v|}$, the first- and second-order moments of multivariate orthonormal polynomials, respectively, are
\begin{equation}
\mathbb{E}\left[ \Psi_{u,\mathbf{j}_u}(\mathbf{X}_u) \right] = 0
\end{equation}
and
\begin{equation}
\mathbb{E}\left[ \Psi_{u,\mathbf{j}_u}(\mathbf{X}_u) \Psi_{v,\mathbf{k}_v}(\mathbf{X}_v) \right] =
\begin{cases}
1, & u=v,~\mathbf{j}_u=\mathbf{k}_v, \\
0, & \text{otherwise}.
\end{cases}
\label{4.26}
\end{equation}
\label{c2}
\end{corollary}

According to Corollary \ref{c2}, the statistical properties of random orthonormal polynomials for independent random variables, especially the covariances between $\Psi_{u,\mathbf{j}_u}(\mathbf{X}_u)$ and $\Psi_{v,\mathbf{k}_v}(\mathbf{X}_v)$, simplify greatly.  These simplified results, readily exploited by the existing PDD \cite{rahman18}, are no longer valid for dependent variables.  In contrast, Corollary \ref{c1} is a general result, is meant for dependent variables, and will be needed in the next section.

\section{Generalized polynomial dimensional decomposition}
The GPDD of a random variable $y(\mathbf{X}) \in L^2(\Omega,\mathcal{F},\mathbb{P})$ is simply the expansion of $y(\mathbf{X})$ with respect to a complete, hierarchically ordered, orthogonal polynomial basis of $L^2(\Omega,\mathcal{F},\mathbb{P})$. A preliminary, less sophisticated version of GPDD was sketched out in a prior work by the author \cite{rahman14b}.  Here, a more rigorous treatment of GPDD entailing dimension-wise splitting of polynomial spaces and functional analysis is presented.  This version is new, has not been published elsewhere, and is, therefore, formally presented here as Theorem \ref{t1}.

\begin{theorem}
\label{t1}
Let $\mathbf{X}:=(X_{1},\ldots,X_{N})^T$ be a vector of $N \in \mathbb{N}$ input random variables fulfilling Assumption \ref{a1}. For $\emptyset \ne u \subseteq \{1,\ldots,N\}$ and $\mathbf{X}_u:=(X_{i_1},\ldots,X_{i_{|u|}})^T:(\Omega^u,\mathcal{F}^u)\to(\mathbb{A}^u,\mathcal{B}^u)$, denote by $\{\Psi_{u,\mathbf{j}_u}(\mathbf{x}_u): \mathbf{j}_u \in \mathbb{N}^{|u|}\}$ the set of standardized multivariate orthogonal polynomials consistent with the probability measure $f_{\mathbf{X}_u}d\mathbf{x}_u$.  Then
\vspace{0.05in}
\begin{enumerate}[{$(1)$}]
\item
for any random variable $y(\mathbf{X}) \in L^2(\Omega, \mathcal{F}, \mathbb{P})$ there exists a Fourier-like series in multivariate orthogonal polynomials in $\mathbf{X}$, referred to as the GPDD of
\begin{equation}
\begin{array}{rcl}
y(\mathbf{X}) & \sim  &
y_{\emptyset} +
\displaystyle \sum_{\emptyset \ne u\subseteq\{1,\ldots,N\}}
\sum_{l=|u|}^\infty
\sum_{\substack{\mathbf{j}_u \in \mathbb{N}^{|u|} \\ |\mathbf{j}_u| = l}}
C_{u,\mathbf{j}_u} \Psi_{u,\mathbf{j}_u}(\mathbf{X}_u) \\
              &  =    &
y_\emptyset + \displaystyle \sum_{\emptyset \ne u\subseteq\{1,\ldots,N\}}
\sum_{\mathbf{j}_u \in \mathbb{N}^{|u|}} C_{u,\mathbf{j}_u} \Psi_{u,\mathbf{j}_u}(\mathbf{X}_u),
\end{array}
\label{5.1}
\end{equation}
where the zero-variate Fourier coefficient $y_{\emptyset} \in \mathbb{R}$ is defined by
\begin{equation}
y_{\emptyset} :=\mathbb{E}\left[ y(\mathbf{X})\right]:=
\int_{\mathbb{A}^N} y(\mathbf{x}) f_{\mathbf{X}}(\mathbf{x}) d\mathbf{x}
\label{5.2}
\end{equation}
and the $|u|$-variate Fourier coefficients $C_{u,\mathbf{j}_u} \in \mathbb{R}$ satisfy the infinite-dimensional linear system
\begin{equation}
\sum_{\emptyset \ne v \subseteq \{1,\ldots,N\}}
\sum_{\mathbf{k}_v \in \mathbb{N}^{|v|}}
C_{v,\mathbf{k}_v}J_{u,\mathbf{j}_u;v,\mathbf{k}_v}=I_{u,\mathbf{j}_u},\;
\emptyset \ne u \subseteq \{1,\ldots,N\},\; \mathbf{j}_u \in \mathbb{N}^{|u|},
\label{5.3}
\end{equation}
with the integrals
\begin{subequations}
\begin{align}
I_{u,\mathbf{j}_u} & :=\mathbb{E}\left[ y(\mathbf{X})\Psi_{u,\mathbf{j}_u}(\mathbf{X}_u) \right]:=
\int_{\mathbb{A}^{N}}y(\mathbf{x})\Psi_{u,\mathbf{j}_u}(\mathbf{x}_u)f_{\mathbf{X}}(\mathbf{x}) d\mathbf{x},
\label{5.4a} \\
J_{u,\mathbf{j}_u;v,\mathbf{k}_v} & :=\mathbb{E}
\left[ \Psi_{u,\mathbf{j}_u}(\mathbf{X}_v)\Psi_{v,\mathbf{k}_v}(\mathbf{X}_v) \right]:=
\int_{\mathbb{A}^{N}}\Psi_{u,\mathbf{j}_u}(\mathbf{x}_u)\Psi_{v,\mathbf{k}_v}(\mathbf{x}_v)
f_{\mathbf{X}}(\mathbf{x}) d\mathbf{x};
\label{5.4b}
\end{align}
\label{5.4}
\end{subequations}
and
\vspace{0.05in}
\item
the GPDD of $y(\mathbf{X}) \in L^2(\Omega, \mathcal{F}, \mathbb{P})$ converges to $y(\mathbf{X})$ in mean-square, that is, for \[
y_{S,m}(\mathbf{X}):=
y_\emptyset +
\sum_{\substack{\emptyset \ne u \subseteq \{1,\ldots,N\} \\ 1 \le |u| \le S}}
\sum_{\substack{\mathbf{j}_u \in \mathbb{N}^{|u|} \\ |u| \le |\mathbf{j}_u| \le m}}
C_{u,\mathbf{j}_u} \Psi_{u,\mathbf{j}_u}(\mathbf{X}_u),
\]
where $1 \le S \le N$ and $|u| \le m < \infty$ are integers,
\[
\lim_{S \to N,\;m \to \infty} \mathbb{E}\left[ y_{S,m}^2(\mathbf{X}) \right] =
\mathbb{E}\left[ y^2(\mathbf{X})\right];
\]
converges in probability, that is, for any $\epsilon>0$,
\[
\lim_{S \to N,\;m \to \infty} \mathbb{P}\left( \left|y_{S,m}(\mathbf{X})-y(\mathbf{X})\right| > \epsilon \right) = 0;
\]
and converges in distribution, that is, for any $\xi \in \mathbb{R}$,
\[
\lim_{S \to N,\;m \to \infty} F_{S,m}(\xi) = F(\xi)
\]
such that $F_{S,m}(\xi):=\mathbb{P}(y_{S,m}(\mathbf{X}) \le \xi)$ and $F(\xi):=\mathbb{P}(y(\mathbf{X}) \le \xi)$ are continuous distribution functions.
\end{enumerate}
\end{theorem}

\begin{proof}
Under Assumption \ref{a1}, a complete infinite set of multivariate orthogonal polynomials in $\mathbf{x}_u$ consistent with the probability measure $f_{\mathbf{X}_u}(\mathbf{x}_u)d\mathbf{x}_u$ exists.  From Proposition \ref{p2} and the fact that standardization is merely scaling, the set of polynomials from the union-sum collection
\begin{equation}
\displaystyle \boldsymbol{1} \oplus  \bigcup_{\emptyset \ne u \subseteq \{1,\ldots,N\}} \bigoplus_{l=|u|}^\infty
\text{span}\{ \Psi_{u,\mathbf{j}_u}(\mathbf{x}_u): |\mathbf{j}_u|=l, \mathbf{j}_u \in \mathbb{N}^{|u|} \}
= \Pi^N
\label{5.5}
\end{equation}
is also dense in $L^2(\mathbb{A}^N,\mathcal{B}^{N},f_{\mathbf{X}}d\mathbf{x})$.  Equivalently, the set of random polynomials
\begin{equation}
\displaystyle \boldsymbol{1} \oplus  \bigcup_{\emptyset \ne u \subseteq \{1,\ldots,N\}} \bigoplus_{l=|u|}^\infty
\text{span}\{ \Psi_{u,\mathbf{j}_u}(\mathbf{X}_u): |\mathbf{j}_u|=l, \mathbf{j}_u \in \mathbb{N}^{|u|} \}
\label{5.5b}
\end{equation}
is dense in $L^2(\Omega, \mathcal{F}, \mathbb{P})$ as well.  Therefore, any random variable $y(\mathbf{X})\in L^2(\Omega, \mathcal{F}, \mathbb{P})$ can be expanded as shown in \eqref{5.1}.  Combining the two inner sums of the expansion forms the equality in the second line of \eqref{5.1}.

From the denseness, every element of $L^2(\Omega, \mathcal{F}, \mathbb{P})$ is a limit point of the set of random polynomials in \eqref{5.5b}. Therefore, the infinite series in \eqref{5.1} converges to $y(\mathbf{X})$ in mean-square, that is,
\begin{equation}
\mathbb{E} \Biggl[ \Biggl\{
\displaystyle y_\emptyset +
\sum_{\emptyset \ne u \subseteq \{1,\ldots,N\}}
\sum_{\mathbf{j}_u \in \mathbb{N}^{|u|}}
C_{u,\mathbf{j}_u} \Psi_{u,\mathbf{j}_u}(\mathbf{X}_u)
\Biggr\}^2
\Biggl]
=
\mathbb{E} \left[ y^2(\mathbf{X}) \right],
\label{5.5f}
\end{equation}
which is similar to the Parseval identity \cite{courant66} for a multivariate orthonormal system. Indeed, GPDD converges in mean-square to the correct limit.

Furthermore, as GPDD converges in mean-square, it does so in probability.  Moreover, as the expansion converges in probability, it also converges in distribution.

Finally, to find the Fourier coefficients, define a second moment
\begin{equation}
e_{\text{GPDD}}:= \mathbb{E}\Biggl[
\Biggl\{
y(\mathbf{X}) -
y_\emptyset - \displaystyle \sum_{\emptyset \ne v\subseteq\{1,\ldots,N\}}
\sum_{\mathbf{k}_v \in \mathbb{N}^{|v|}} C_{v,\mathbf{k}_v} \Psi_{v,\mathbf{k}_v}(\mathbf{X}_v)
\Biggr \}^2
\Biggr]
\label{5.6}
\end{equation}
of the difference between $y(\mathbf{X})$ and its full GPDD.  Differentiate both sides of \eqref{5.6} with respect to $y_\emptyset$ and $C_{u,\mathbf{j}_u}$, $\emptyset \ne u \subseteq \{1,\ldots,N\}$, $\mathbf{j}_u \in \mathbb{N}^{|u|}$, to write
\begin{equation}
\begin{array}{rcl}
\displaystyle \frac{\partial e_{\text{GPDD}}}{\partial y_\emptyset}
& = &
\displaystyle
\frac{\partial}{\partial y_\emptyset}
\mathbb{E} \Biggl[
\Biggl\{
y(\mathbf{X}) - y_\emptyset - \displaystyle \sum_{\emptyset \ne v\subseteq\{1,\ldots,N\}}
\sum_{\mathbf{k}_v \in \mathbb{N}^{|v|}} C_{v,\mathbf{k}_v} \Psi_{v,\mathbf{k}_v}(\mathbf{X}_v)
\Biggr\}^2
\Biggl]
 \\
&  = &
\displaystyle
\mathbb{E} \Biggl[
\frac{\partial}{\partial y_\emptyset} \Biggl\{
y(\mathbf{X}) -  y_\emptyset - \displaystyle \sum_{\emptyset \ne v\subseteq\{1,\ldots,N\}}
\sum_{\mathbf{k}_v \in \mathbb{N}^{|v|}} C_{v,\mathbf{k}_v} \Psi_{v,\mathbf{k}_v}(\mathbf{X}_v)
  \Biggr\}^2
\Biggr]
 \\
&  = &
\displaystyle
 2 \mathbb{E} \Biggl[
 \Biggl\{
 y_\emptyset + \displaystyle \sum_{\emptyset \ne v\subseteq\{1,\ldots,N\}}
\sum_{\mathbf{k}_v \in \mathbb{N}^{|v|}} C_{v,\mathbf{k}_v} \Psi_{v,\mathbf{k}_v}(\mathbf{X}_v) - y(\mathbf{X}) \Biggr\}
\times 1
 \Biggr]
 \\
&  = &
\displaystyle
2 \left\{ y_\emptyset - \mathbb{E}\left[ y(\mathbf{X})\right] \right\}
\end{array}
\label{5.7}
\end{equation}
and
\begin{equation}
\begin{array}{rcl}
\displaystyle \frac{\partial e_{\text{GPDD}}}{\partial C_{u,\mathbf{j}_u}}
& = &
\displaystyle
\frac{\partial}{\partial C_{u,\mathbf{j}_u}}
\mathbb{E} \Biggl[
\Biggl\{
y(\mathbf{X}) - y_\emptyset - \displaystyle \sum_{\emptyset \ne v\subseteq\{1,\ldots,N\}}
\sum_{\mathbf{k}_v \in \mathbb{N}^{|v|}} C_{v,\mathbf{k}_v} \Psi_{v,\mathbf{k}_v}(\mathbf{X}_v)
\Biggr\}^2
\Biggl]
 \\
&  = &
\displaystyle
\mathbb{E} \Biggl[
\frac{\partial}{\partial C_{u,\mathbf{j}_u}} \Biggl\{
y(\mathbf{X}) -  y_\emptyset - \displaystyle \sum_{\emptyset \ne v\subseteq\{1,\ldots,N\}}
\sum_{\mathbf{k}_v \in \mathbb{N}^{|v|}} C_{v,\mathbf{k}_v} \Psi_{v,\mathbf{k}_v}(\mathbf{X}_v)
  \Biggr\}^2
\Biggr]
 \\
&  = &
\displaystyle
 2 \mathbb{E} \Biggl[
 \Biggl\{
 y_\emptyset +  \displaystyle \sum_{\emptyset \ne v\subseteq\{1,\ldots,N\}}
\sum_{\mathbf{k}_v \in \mathbb{N}^{|v|}} C_{v,\mathbf{k}_v} \Psi_{v,\mathbf{k}_v}(\mathbf{X}_v) - y(\mathbf{X}) \Biggr\}
\Psi_{u,\mathbf{j}_u}(\mathbf{X}_u)
 \Biggr]
 \\
&  = &
\displaystyle
2\Biggl\{
\sum_{\emptyset \ne v \subseteq \{1,\ldots,N\}}
\sum_{\mathbf{k}_v \in \mathbb{N}^{|v|}}\!\!
C_{v,\mathbf{k}_v} \mathbb{E}\left[ \Psi_{u,\mathbf{j}_u}(\mathbf{X}_u)\Psi_{v,\mathbf{k}_v}(\mathbf{X}_v)\right]-
\mathbb{E}\left[ y(\mathbf{X})\Psi_{u,\mathbf{j}_u}(\mathbf{X}_u)\right]
\Biggr\}
\\
&  = &
\displaystyle
2\Biggl\{
\sum_{\emptyset \ne v \subseteq \{1,\ldots,N\}}
\sum_{\mathbf{k}_v \in \mathbb{N}^{|v|}}
C_{v,\mathbf{k}_v} J_{u,\mathbf{j}_u;v,\mathbf{k}_v} - I_{u,\mathbf{j}_u}
\Biggr\}.
\end{array}
\label{5.8}
\end{equation}
Here, the second, third, and fourth lines of both \eqref{5.7} and \eqref{5.8} are obtained by interchanging the differential and expectation operators; performing the differentiation; and swapping the expectation and summation operators and then applying Corollary \ref{c1}, respectively.  The interchanges are permissible as the infinite sum is convergent as demonstrated in the preceding paragraph.  The last line of \eqref{5.8} is formed using Corollary \ref{c1} and definitions of the two integrals in \eqref{5.4a} and \eqref{5.4b}.  Setting ${\partial e_{\text{GPDD}}}/{\partial y_\emptyset}=0$ in \eqref{5.7} and ${\partial e_{\text{GPDD}}}/{\partial C_{u,\mathbf{j}_u}}=0$ in \eqref{5.8} yields \eqref{5.2} and \eqref{5.3}, respectively, completing the proof.
\end{proof}

It should be emphasized that the function $y$ must be square-integrable for the mean-square and other convergences to hold.  However, the rate of convergence depends on the smoothness of the function.  The smoother the function, the faster the convergence. If the function is a polynomial, then its GPDD exactly reproduces the function.  These results can be easily proved using classical approximation theory.

It is important to recognize that the definitions of the integrals $I_{u,\mathbf{j}_u}$ and $J_{u,\mathbf{j}_u;v,\mathbf{k}_v}$ in \eqref{5.4a} and \eqref{5.4b} are not identical to those presented in the author's past work \cite{rahman14b}.  There, the aforementioned integrals were defined as
\footnote{Strictly speaking, the integrals in \cite{rahman14b} were defined using the support of the density function of $\mathbf{X}$ as $\mathbb{R}^N$.  The extension to a support $\mathbb{A}^N \subseteq \mathbb{R}^N$ should follow readily.}
\begin{equation}
\bar{I}_{u,\mathbf{j}_u}  :=
\int_{\mathbb{A}^{N}}y(\mathbf{x})\Psi_{u,\mathbf{j}_u}(\mathbf{x}_u)
f_{\mathbf{X}_u}(\mathbf{x}_u)f_{\mathbf{X}_{-u}}(\mathbf{x}_{-u}) d\mathbf{x}
\label{5.8a}
\end{equation}
and
\begin{equation}
\bar{J}_{u,\mathbf{j}_u;v,\mathbf{k}_v} :=\mathbb{E}
\left[ \Psi_{u,\mathbf{j}_u}(\mathbf{X}_v)\Psi_{v,\mathbf{k}_v}(\mathbf{X}_v) \right]:=
\int_{\mathbb{A}^{N}}\Psi_{u,\mathbf{j}_u}(\mathbf{x}_u)\Psi_{v,\mathbf{k}_v}(\mathbf{x}_v)
f_{\mathbf{X}_u}(\mathbf{x}_u)f_{\mathbf{X}_{v \cap -u}}(\mathbf{x}_{v \cap -u}) d\mathbf{x}.
\label{5.8b}
\end{equation}
Clearly, the density functions in \eqref{5.8a} and \eqref{5.8b} are different than those in \eqref{5.4a} and \eqref{5.4b}.  As a result, the former integrals cannot be interpreted as expectations as the latter integrals.  The difference arises due to distinct perspectives involved in deriving the final expressions of the expansion coefficients.  Consequently, the resultant linear systems from this and past works are also different, although both lead to the calculation of the expansion coefficients. Indeed, the linear system \eqref{5.3} is new and has not been published elsewhere.  More importantly, the new definitions of the two integrals in \eqref{5.4a} and \eqref{5.4b} and the linear system \eqref{5.3} enable a decoupling procedure for calculating the expansion coefficients efficiently, to be discussed next.

The system \eqref{5.3} can be broken down further as many of the integral coefficients, that is, $J_{u,\mathbf{j}_u;v,\mathbf{k}_v}$, vanishes, according to Corollary \ref{c1}.  Indeed, as $J_{u,\mathbf{j}_u;v,\mathbf{k}_v}=0$ for $|\mathbf{j}_u|\ne |\mathbf{k}_v|$, \eqref{5.3} is actually an infinite system of uncoupled finite-dimensional linear systems. The Fourier coefficients interact with each other only for a specific degree -- a consequence of employing orthogonal polynomial basis.  By reshuffling the coefficients according to the degree $1 \le l < \infty$, the GPDD in \eqref{5.1} can also be written as
\begin{equation}
y(\mathbf{X}) =
\displaystyle
y_\emptyset +
\sum_{l \in \mathbb{N}}
\sum_{\substack{\emptyset \ne u \subseteq \{1,\ldots,N\} \\ 1 \le |u| \le \min(N,l)}}
\sum_{\substack{\mathbf{j}_u \in \mathbb{N}^{|u|} \\ |\mathbf{j}_u|=l}}
C_{u,\mathbf{j}_u} \Psi_{u,\mathbf{j}_u}(\mathbf{X}_u).
\label{5.9}
\end{equation}
For each degree $l$, there are
\begin{equation}
Q_{N,l}=
\displaystyle
\sum_{s=1}^{\min(N,l)} \binom{N}{s}\binom{l-1}{s-1} < \infty
\label{5.10}
\end{equation}
number of Fourier coefficients $C_{u,\mathbf{j}_u}$, $1 \le |u| \le \min(N,l)$, $|\mathbf{j}_u|=l$, which satisfy the $Q_{N,l}\times Q_{N,l}$ linear system
\begin{equation}
\displaystyle
\sum_{\substack{\emptyset \ne v \subseteq \{1,\ldots,N\} \\ 1 \le |v| \le \min(N,l)}}
\sum_{\substack{\mathbf{k}_v \in \mathbb{N}^{|v|} \\ |\mathbf{k}_v|=|\mathbf{j}_u|}}
C_{v,\mathbf{k}_v}J_{u,\mathbf{j}_u;v,\mathbf{k}_v}=I_{u,\mathbf{j}_u},~
1 \le |u| \le \min(N,l),~
|\mathbf{j}_u|=l.
\label{5.11}
\end{equation}
The coefficient matrix in the matrix form of \eqref{5.11} comprises expectations of the product of two orthogonal polynomials.  In other words, the coefficient matrix is a Gram matrix, which is positive-definite and hence invertible. Indeed, \eqref{5.9}, \eqref{5.10}, and \eqref{5.11} facilitate a systematic and computationally efficient procedure for determining the Fourier coefficients of GPDD.

\begin{corollary}
Let $\mathbf{X}=(X_1,\ldots,X_N)^T$ be a vector of independent, but not necessarily identical, input random variables, satisfying Assumption \ref{a1}.  Denote by $f_{X_i}(x_i)$, $i=1,\ldots,N$, the marginal density function of the $i$th random variable $X_i$ and by $\Psi_{\{i\},j_i}(x_i)$ the $j_i$th-degree univariate orthonormal polynomial in $x_i$, which is obtained consistent with the probability measure $f_{X_i}(x_i)dx_i$. Then the proposed GPDD reduces to the existing PDD, yielding
\begin{equation}
y(\mathbf{X}) \sim
y_\emptyset +
\displaystyle
\sum_{\emptyset \ne u\subseteq\{1,\ldots,N\}}
\sum_{\mathbf{j}_u \in \mathbb{N}^{|u|}}
C_{u,\mathbf{j}_u} \prod_{p=1}^{|u|} \Psi_{\{i_p\},j_{i_p}}(X_{i_p})
\label{5.12}
\end{equation}
with the Fourier coefficients
\[
y_{\emptyset} =\mathbb{E}\left[ y(\mathbf{X}) \right] =
\int_{\mathbb{A}^N} y(\mathbf{x}) \prod_{i=1}^N f_{X_i}(x_i) dx_i
\]
and
\[
C_{u,\mathbf{j}_u} = \mathbb{E}\left[ y(\mathbf{X}) \prod_{p=1}^{|u|} \Psi_{\{i_p\},j_{i_p}}(X_{i_p}) \right]:=
\int_{\mathbb{A}^N} y(\mathbf{x}) \prod_{p=1}^{|u|} \Psi_{\{i_p\},j_{i_p}}(x_{i_p}) \prod_{i=1}^N f_{X_i}(x_i) dx_i.
\]
\label{c4}
\end{corollary}

\begin{proof}
For independent input variables, one has $f_{\mathbf{X}}(\mathbf{x})=\prod_{i=1}^N f_{X_i}(x_i)$, leading to
$\Psi_{u,\mathbf{j}_u}(\mathbf{X}_u)=\prod_{p=1}^{|u|} \Psi_{\{i_p\},j_{i_p}}(X_{i_p})$, as per Corollary \ref{c2}.  Also, the integral $J_{u,\mathbf{j}_u;v,\mathbf{k}_v}$ vanishes whenever $u\ne v$.  Subsequently, the result of Corollary \ref{c4} follows readily.
\end{proof}

Note that the infinite series in \eqref{5.1} or \eqref{5.12} does not necessarily converge almost surely to $y(\mathbf{X})$. Furthermore, there is no guarantee that the moments of PDD or GPDD of order larger than two will converge.  These are due to the fundamental limitations of Fourier or Fourier-like series.

\subsection{Connection to the generalized ADD}
It is important to point out the relation between GPDD and the generalized ADD discussed in Section 3. For instance, comparing \eqref{3.6} and \eqref{4.10} yields the closure of an orthogonal decomposition of
\begin{equation}
\mathcal{W}_u =
\overline{\bigoplus_{l=|u|}^\infty \mathcal{Z}_l^u}
\end{equation}
into polynomial spaces $\mathcal{Z}_l^u$, $|u| \le l < \infty$, resulting in
\begin{equation}
y_u(\mathbf{X}_u) \sim
\displaystyle
\sum_{\mathbf{j}_u \in \mathbb{N}^{|u|}} C_{u,\mathbf{j}_u} \Psi_{u,\mathbf{j}_u}(\mathbf{X}_u).
\label{5.13}
\end{equation}
Indeed, the connection between GPDD and the generalized ADD is clearly palpable, where the former can be viewed as a polynomial variant of the latter.  Here, $C_{u,\mathbf{j}_u} \Psi_{u,\mathbf{j}_u}(\mathbf{X}_u)$ in \eqref{5.1} or \eqref{5.13} represents a $|u|$-variate, $|\mathbf{j}_u|$th-order GPDD component function of $y(\mathbf{X})$, describing the $|\mathbf{j}_u|$th-order polynomial approximation of the $|u|$-variate component function $y_u(\mathbf{X}_u)$ of the generalized ADD.

Moreover, given the second-moment properties of multivariate orthogonal polynomials in Proposition \ref{p3} or Corollary \ref{c1}, it is easy to see why the second-moment properties of $y_u(\mathbf{X}_u)$, that is, \eqref{3.3} and \eqref{3.4}, are automatically satisfied when $y_u(\mathbf{X}_u)$ is expanded as in \eqref{5.13}. Therefore, GPDD inherits all desirable properties of the generalized ADD -- an important, fundamental requirement for any refinement of the latter.

\subsection{Truncation}
The full GPDD contains an infinite number of orthogonal polynomials or coefficients. In practice, the number must be finite, meaning that GPDD must be truncated. However, there are multiple ways to perform the truncation. A straightforward approach adopted in this work entails (1) keeping all polynomials in at most $0 \le S \le N$ variables, thereby retaining the degrees of interaction among input variables less than or equal to $S$, and (2) preserving polynomial expansion orders (total) less than or equal to $S \le m < \infty$.  The result is an $S$-variate, $m$th-order GPDD approximation\footnote{The nouns \emph{degree} and \emph{order} associated with GPDD or orthogonal polynomials are used synonymously in the paper.}
\begin{equation}
\begin{array}{rcl}
y_{S,m}(\mathbf{X}) & = &
\displaystyle
y_\emptyset +
\sum_{s=1}^S~
\sum_{l=s}^m~
\sum_{\substack{\emptyset \ne u \subseteq \{1,\ldots,N\} \\ |u| = s}}
\sum_{\substack{\mathbf{j}_u \in \mathbb{N}^{|u|} \\ |\mathbf{j}_u| = l}}
C_{u,\mathbf{j}_u} \Psi_{u,\mathbf{j}_u}(\mathbf{X}_u) \\
                    & = &
\displaystyle
y_\emptyset +
\sum_{\substack{\emptyset \ne u \subseteq \{1,\ldots,N\} \\ 1 \le |u| \le S}}
\sum_{\substack{\mathbf{j}_u \in \mathbb{N}^{|u|} \\ |u| \le |\mathbf{j}_u| \le m}}
C_{u,\mathbf{j}_u} \Psi_{u,\mathbf{j}_u}(\mathbf{X}_u)
\end{array}
\label{5.14}
\end{equation}
of $y(\mathbf{X})$, containing
\begin{equation}
{\displaystyle
L_{S,m} = 1 + \sum_{s=1}^S \binom{N}{s} \binom{m}{s}
}
\label{5.15}
\end{equation}
number of Fourier coefficients including $y_\emptyset$.  It is important to clarify a few things about the truncated GPDD proposed.  First, the truncation with respect to the polynomial expansion order in \eqref{5.14} is related to the total degree index set
\[
\left\{ \mathbf{j}_u \in \mathbb{N}^{|u|}:\sum_{p=1}^{|u|} j_{i_p} \le m \right\}.
\]
Other kinds of truncation employ the tensor product and hyperbolic cross index sets, to name just two.  The total degree and tensor product index sets are common choices, although the latter suffers from the curse of dimensionality, making it impractical for high-dimensional problems.  The hyperbolic cross index set is a relatively new idea and has yet to receive widespread attention.  All of these choices and possibly others, including their anisotropic versions, can be used for truncating GPDD.  In this work, however, only the total degree index set is used for the GPDD approximation. Second, the right side of \eqref{5.14} contains sums of at most $S$-dimensional orthogonal polynomials, representing at most $S$-variate GPDD component functions of $y$. Therefore, the term ``$S$-variate'' used for the GPDD approximation should be interpreted in the context of including at most $S$-degree interaction of input variables, even though $y_{S,m}$ is strictly an $N$-variate function. Third, when $S=0$, $y_{0,m}=y_\emptyset$ for any $m$ as the outer sums of \eqref{5.14} vanish.  Finally, when $S \to N$ and $m \to \infty$, $y_{S,m}$ converges to $y$ in mean-square sense, generating a hierarchical and convergent sequence of GPDD approximations.

The motivation behind generalized ADD- and GPDD-derived approximations is the following.  In a practical setting, the function $y(\mathbf{X})$, fortunately, has an effective dimension much lower than $N$, meaning that the right side of \eqref{3.1} can be effectively approximated by a sum of lower-dimensional component functions $y_{u}$, $|u|\ll N$, but still maintaining all random variables $\mathbf{X}$ of a high-dimensional stochastic problem. For instance, an $S$-variate, $m$th-order GPDD approximation $y_{S,m}(\mathbf{X})$ is generated, where $0 \le S \le N$ and $S \le m <\infty$ define the largest degree of interactions among input variables and the largest order of orthogonal polynomials retained in a concomitant truncation. The approximation is grounded on a fundamental conjecture known to be true in many real-world applications: given a high-dimensional function $y$, its  $|u|$-variate, $|\mathbf{j}_u|$th-order GPDD component function $C_{u,\mathbf{j}_u} \Psi_{u,\mathbf{j}_u}(\mathbf{X}_u)$ decays rapidly with respect to $|u|$ and $|\mathbf{j}_u|$, leading to an accurate low-variate, low-order approximation of $y$. From \eqref{5.15}, the computational complexity of a truncated GPDD is polynomial, as opposed to exponential, thereby alleviating the curse of dimensionality to a substantial extent.

It is natural to ask about the approximation quality of \eqref{5.14}.  Since the set of polynomials from \eqref{5.5} is complete in $L^2(\mathbb{A}^N,\mathcal{B}^{N},f_{\mathbf{X}}d\mathbf{x})$, the truncation error $y(\mathbf{X})-y_{N,m}(\mathbf{X})$ is orthogonal to any element of the subspace from which $y_{N,m}(\mathbf{X})$ is chosen, as demonstrated below.

\begin{proposition}
Let
\begin{equation}
\Pi_{S,m}^N :=
\displaystyle
\boldsymbol{1} \oplus
\bigcup_{\substack{\emptyset \ne u \subseteq \{1,\ldots,N\} \\ 1 \le |u| \le S} }
\bigoplus_{\substack{\mathbf{j}_u \in \mathbb{N}^{|u|} \\ |u| \le |\mathbf{j}_u| \le m}}
\text{span}\{ \Psi_{u,\mathbf{j}_u}(\mathbf{X}_u): \mathbf{j}_u \in \mathbb{N}^{|u|} \} \subseteq L^2(\Omega,\mathcal{F},\mathbb{P})
\label{5.16}
\end{equation}
be a subspace comprising all polynomials in $\mathbf{X}$ with the degree of interaction at most $1 \le S \le N$ and order at most $S \le m < \infty$, including constants.  For any $y(\mathbf{X}) \in L^2(\Omega,\mathcal{F},\mathbb{P})$, denote by $y_{S,m}(\mathbf{X})$ and $y_{N,m}(\mathbf{X})$ its $S$-variate, $m$th-order and $N$-variate, $m$th-order GPDD approximations, respectively.  Then the truncation error $y(\mathbf{X})-y_{N,m}(\mathbf{X})$ is orthogonal to the subspace $\Pi_{N,m}^N \subseteq L^2(\Omega,\mathcal{F},\mathbb{P})$. Moreover, $\mathbb{E}[\{y(\mathbf{X})-y_{S,m}(\mathbf{X})\}^2] \to 0$ as $S \to N$ and $m \to \infty$.
\label{p5}
\end{proposition}

\begin{proof}
Let
\begin{equation}
\bar{y}_{N,m}(\mathbf{X}) :=
\bar{y}_\emptyset +
\sum_{\emptyset \ne v \subseteq \{1,\ldots,N\}}
\sum_{\substack{\mathbf{k}_v \in \mathbb{N}^{|v|} \\ |v| \le |\mathbf{k}_v| \le m}}
\bar{C}_{v,\mathbf{k}_v} \Psi_{v,\mathbf{k}_v}(\mathbf{X}_v),
\end{equation}
with arbitrary expansion coefficients $\bar{y}_\emptyset$ and $\bar{C}_{v,\mathbf{k}_v}$, be any element of the subspace $\Pi_{N,m}^N$ of $L^2(\Omega,\mathcal{F},\mathbb{P})$ described by \eqref{5.16} for $S=N$.  Then
\begin{equation}
\begin{array}{rcl}
&   &
{\displaystyle
\mathbb{E} \left[ \{y(\mathbf{X})-y_{N,m}(\mathbf{X})\}\bar{y}_{N,m}(\mathbf{X}) \right]
}  \\
& = &
\displaystyle
\mathbb{E} \Biggl[ \Biggl\{
\sum_{\emptyset \ne u \subseteq \{1,\ldots,N\} }
\sum_{\substack{\mathbf{j}_u \in \mathbb{N}^{|u|} \\ m+1 \le |\mathbf{j}_u| < \infty}}
\!\!\!\!C_{u,\mathbf{j}_u} \Psi_{u,\mathbf{j}_u}(\mathbf{X}_u)  \Biggr\}
\Biggl\{
\bar{y}_\emptyset +
\sum_{\emptyset \ne v \subseteq \{1,\ldots,N\}  }
\sum_{\substack{\mathbf{k}_v \in \mathbb{N}^{|v|} \\ |v| \le |\mathbf{k}_v| \le m}}
\!\!\bar{C}_{v,\mathbf{k}_v} \Psi_{v,\mathbf{k}_v}(\mathbf{X}_v)
\Biggr\}
\Biggr]
\\
&  = &
0,
\end{array}
\end{equation}
where the last line follows from a \emph{zero} result of Corollary \ref{c1}, proving the first part of the proposition.  For the latter part, the Pythagoras theorem yields
\begin{equation}
\mathbb{E}[\{y(\mathbf{X})-y_{N,m}(\mathbf{X})\}^2] + \mathbb{E}[y_{N,m}^2(\mathbf{X})] =
\mathbb{E}[y^2(\mathbf{X})].
\end{equation}
Then
\begin{equation}
\begin{array}{rcl}
\displaystyle
\lim_{S \to N,m \to \infty} \mathbb{E}\left[\{y(\mathbf{X})-y_{S,m}(\mathbf{X})\}^2\right]
&  =  &
\displaystyle
\lim_{m \to \infty} \mathbb{E}\left[\{y(\mathbf{X})-y_{N,m}(\mathbf{X})\}^2\right] \\
&  =  &
\displaystyle
\lim_{m \to \infty} \left( \mathbb{E}\left[y^2(\mathbf{X})\right]-
\mathbb{E}\left[y_{N,m}^2(\mathbf{X})\right] \right)\\
&  =  &
\displaystyle
\mathbb{E}\left[y^2(\mathbf{X})\right] -
\lim_{S \to N,m \to \infty} \mathbb{E}\left[y_{S,m}^2(\mathbf{X})\right] \\
&  =  &
0,
\end{array}
\end{equation}
where the second line uses the result of the Pythagoras theorem; and the equality to \emph{zero} in the last line stems from Theorem \ref{t1}, which says that $\mathbb{E}[y_{S,m}^2(\mathbf{X})] \to \mathbb{E}[y^2(\mathbf{X})]$ as $S \to N$ and $m \to \infty$.
\end{proof}

The second part of Proposition \ref{p5} entails $L^2$ convergence, which is the same as the mean-square convergence described in Theorem \ref{t1}.  However, an alternative route is chosen for the proof of the proposition.

\begin{proposition}
Let $\Pi_{S,m}^N$ be as defined in Proposition \ref{p5}.  Then the $S$-variate, $m$th-order GPDD approximation $y_{S,m}(\mathbf{X})$ of $y(\mathbf{X}) \in L^2(\Omega,\mathcal{F},\mathbb{P})$ is the best approximation in the sense that
\begin{equation}
\displaystyle
\mathbb{E}\left[ \{y(\mathbf{X})-y_{S,m}(\mathbf{X}) \}^2 \right] =
\inf_{\bar{y}_{S,m} \in \Pi_{S,m}^N}
\mathbb{E}\left[ \{ y(\mathbf{X})-\bar{y}_{S,m}(\mathbf{X}) \}^2 \right].
\label{5.15e}
\end{equation}
\label{p6}
\end{proposition}

\begin{proof}
Let
\begin{equation}
\bar{y}_{S,m}(\mathbf{X}) :=
\bar{y}_\emptyset +
\sum_{\substack{\emptyset \ne u \subseteq \{1,\ldots,N\}\\ 1 \le |u| \le S} }
\sum_{\substack{\mathbf{j}_u \in \mathbb{N}^{|u|} \\ |u| \le |\mathbf{j}_u| \le m}}
\bar{C}_{v,\mathbf{k}_v} \Psi_{v,\mathbf{k}_v}(\mathbf{X}_v),
\label{5.15e}
\end{equation}
with arbitrary expansion coefficients $\bar{y}_\emptyset$ and $\bar{C}_{u,\mathbf{j}_u}$, be any element of the subspace $\Pi_{S,m}^N \subseteq L^2(\Omega,\mathcal{F},\mathbb{P})$ defined in \eqref{5.16}.  To minimize $\mathbb{E}[ \{y(\mathbf{X})-\bar{y}_{S,m}(\mathbf{X})\}^2]$, the derivatives with respect to the coefficients must be \emph{zero}, that is,
\[
\displaystyle
\frac{\partial}{\partial \bar{y}_\emptyset}
\mathbb{E}\left[ \{y(\mathbf{X})-\bar{y}_{S,m}(\mathbf{X}) \}^2 \right]=
\frac{\partial}{\partial \bar{C}_{u,\mathbf{j}_u}}
\mathbb{E}\left[ \{y(\mathbf{X})-\bar{y}_{S,m}(\mathbf{X}) \}^2 \right]=0.
\]
From the proof of Theorem \ref{t1}, for instance, \eqref{5.7} and \eqref{5.8} and the following text, the derivatives are \emph{zero} only when $\bar{y}_\emptyset=y_\emptyset$ and $\bar{C}_{u,\mathbf{j}_u}=C_{u,\mathbf{j}_u}$, where $y_\emptyset$ and $C_{u,\mathbf{j}_u}$ are the Fourier coefficients of GPDD in \eqref{5.2} and \eqref{5.3}, respectively.
\end{proof}

\subsection{Infinitely many input variables}
In many fields, such as uncertainty quantification, information theory, and stochastic process, functions depending on a countable sequence $\{X_i\}_{i \in \mathbb{N}}$ of input random variables need to be considered \cite{griebel16}.  Under certain assumptions, GPDD is still applicable as in the case of finitely many random variables, as demonstrated by the following proposition.

\begin{proposition}
Let $\{X_i\}_{i \in \mathbb{N}}$ be a countable sequence of input random variables defined on the probability space $(\Omega, \mathcal{F}_\infty, \mathbb{P})$, where $\mathcal{F}_\infty:=\sigma(\{ X_i\}_{i \in \mathbb{N}})$ is the associated $\sigma$-algebra generated.  If the sequence $\{X_i\}_{i \in \mathbb{N}}$ satisfies Assumption \ref{a1}, then the GPDD of $y(\{X_i\}_{i \in \mathbb{N}}) \in L^2(\Omega, \mathcal{F}_\infty, \mathbb{P})$, where $y:\mathbb{A}^{\mathbb{N}} \to \mathbb{R}$, converges to $y(\{X_i\}_{i \in \mathbb{N}})$ in mean-square.  Moreover, the GPDD converges in probability and in distribution.
\label{p5b}
\end{proposition}

\begin{proof}
According to Proposition \ref{p2}, $\Pi^N$ is dense in $L^2(\mathbb{A}^N,\mathcal{B}^N,f_{\mathbf{X}}d\mathbf{x})$ and hence in $L^2(\Omega, \mathcal{F}_N, \mathbb{P})$ for every $N \in \mathbb{N}$, where $\mathcal{F}_N:=\sigma(\{ X_i\}_{i=1}^N)$ is the associated $\sigma$-algebra generated by $\{ X_i\}_{i=1}^N$.  Here, with a certain abuse of notation, $\Pi^N$ is used as a set of polynomial functions of both real variables $\mathbf{x}$ and random variables $\mathbf{X}$.  Now, apply Theorem 3.8 of Ernst \textit{et al}. \cite{ernst12}, which says that if $\Pi^N$ is dense in $L^2(\Omega, \mathcal{F}_N, \mathbb{P})$ for every $N \in \mathbb{N}$, then
\[
\Pi^\infty:= \bigcup_{N=1}^\infty \Pi^N,
\]
a subspace of $L^2(\Omega, \mathcal{F}_\infty, \mathbb{P})$, is also dense in $L^2(\Omega, \mathcal{F}_\infty, \mathbb{P})$.  But, using \eqref{5.5},
\[
\begin{array}{rcl}
\Pi^\infty & = &
\displaystyle
\bigcup_{N=1}^\infty
\boldsymbol{1} \oplus  \bigcup_{\emptyset \ne u \subseteq \{1,\ldots,N\}}~
\bigoplus_{l=|u|}^\infty
\text{span}\{ \Psi_{u,\mathbf{j}_u}: |\mathbf{j}_u|=l, \mathbf{j}_u \in \mathbb{N}^{|u|} \} \\
           & = &
\displaystyle
\boldsymbol{1} \oplus  \bigcup_{\emptyset \ne u \subseteq \mathbb{N}}~
\bigoplus_{l=|u|}^\infty
\text{span}\{ \Psi_{u,\mathbf{j}_u}: |\mathbf{j}_u|=l, \mathbf{j}_u \in \mathbb{N}^{|u|} \},
\end{array}
\]
demonstrating that the set of polynomials from the union-sum in the last line is dense in $L^2(\Omega, \mathcal{F}_\infty, \mathbb{P})$.  Therefore, the GPDD of $y(\{X_i\}_{i \in \mathbb{N}}) \in L^2(\Omega, \mathcal{F}_\infty, \mathbb{P})$ converges to $y(\{X_i\}_{i \in \mathbb{N}})$ in mean-square.  Since the mean-square convergence is stronger than the convergence in probability or in distribution, the latter modes of convergence follow readily.
\end{proof}

\subsection{Comparison with generalized polynomial chaos expansion}
While the paper focuses on a dimension-wise Fourier-like series in orthogonal polynomials, a comparison with competing expansions entailing orthogonal polynomials without dimensional hierarchy should be intriguing.  One such expansion is the generalized polynomial chaos expansion (PCE) or GPCE, developed recently for dependent input random variables \cite{rahman18b}.  It is derived from a degree-wise splitting of the polynomial spaces, so that any square-integrable output random variable $y(\mathbf{X})$ can be expanded as \cite{rahman18b}
\begin{equation}
y(\mathbf{X}) \sim
\sum_{\mathbf{j} \in \mathbb{N}_0^{N}}
C_{\mathbf{j}} \Psi_{\mathbf{j}}(\mathbf{X}),
\label{5.16}
\end{equation}
where $\{\Psi_{\mathbf{j}}(\mathbf{X}): \mathbf{j} \in \mathbb{N}_0^{N}\}$ is an infinite set of measure-consistent multivariate orthonormal polynomials in $\mathbf{X}$ and $C_{\mathbf{j}} \in \mathbb{R}$, $\mathbf{j} \in \mathbb{N}_0^{N}$, are the Fourier coefficients of GPCE. Like GPDD, the GPCE of $y(\mathbf{X}) \in L^2(\Omega, \mathcal{F}, \mathbb{P})$ under Assumption \ref{a1} also converges to $y(\mathbf{X})$ in mean-square, in probability, and in distribution \cite{rahman18b}.  When truncated according to the total degree index set, the $p$th-order GPCE approximation of $y(\mathbf{X})$, where $0 \le p < \infty$, reads
\begin{equation}
y_p(\mathbf{X}) =
\sum_{\substack{\mathbf{j} \in \mathbb{N}_0^{N} \\ 0 \le |\mathbf{j}| \le p}}
C_{\mathbf{j}} \Psi_{\mathbf{j}}(\mathbf{X}).
\label{5.17}
\end{equation}
Clearly, the two infinite series from GPDD and GPCE, defined by \eqref{5.1} and \eqref{5.16}, respectively, are the same or equivalent with respect to their identical second-moment properties.  However, GPDD and GPCE when truncated are not. In fact, two notable observations jump out. First, the terms in the GPCE approximation in \eqref{5.17} are organized strictly with respect to the order of polynomials. In contrast, the GPDD approximation in \eqref{5.14} is structured with respect to both the degree of interaction among random variables and the order of polynomials. Therefore, significant differences may exist regarding the accuracy, efficiency, and convergence properties of their truncated sums.  Second, if a stochastic response is highly nonlinear, but contains rapidly diminishing interactive effects of input random variables -- a premise supported by real-world applications -- the GPDD approximation is expected to be more effective than the GPCE approximation. This is because the lower-variate terms of the GPDD approximation can be just as nonlinear by selecting appropriate values of $m$ in \eqref{5.14}. In contrast, many more terms and Fourier coefficients are required to be included in the GPCE approximation to capture such high nonlinearity.  To better explain this point, a numerical example discussing error analysis of both approximations is illustrated next.

\section{A numerical example}
Consider a polynomial
\begin{equation}
\displaystyle
y(\mathbf{X}) = 10 \left( X_1^6 + X_2^6 + X_3^6 \right) +
\frac{1}{10} \left( X_1 X_2 + X_1 X_3 + X_2 X_3 \right) +
\frac{1}{1000} X_1^2 X_2^2 X_3^2
\label{6.1}
\end{equation}
in three real-valued, dependent random variables $(X_1,X_2,X_3)$, which follow the Dirichlet probability density function
\[
f_{X_{1}X_{2}X_{3}}(x_{1},x_{2},x_{3})=
\begin{cases}
{\displaystyle \frac{{\displaystyle \Gamma\left({\displaystyle \sum_{i=1}^{4}\kappa_{i}+{\displaystyle 2}}\right)}}{{\displaystyle \prod_{i=1}^{4}\Gamma\left(\kappa_{i}+\frac{1}{2}\right)}}{\displaystyle {\displaystyle \left({\displaystyle {\displaystyle \prod_{i=1}^{3}x_{i}^{\kappa_{i}-\frac{1}{2}}}}\right)\left(1-x_1 - x_2 - x_3 \right)^{\kappa_{4}-\frac{1}{2}}},}} & \mathbf{x}=(x_{1},x_{2},x_{3})^{T}\in\mathbb{T}^{3},\\
0, & \text{otherwise},
\end{cases}
\]
on the standard tetrahedron $\mathbb{T}^3:=\{(x_1,x_2,x_3): 0\le x_1,x_2,x_3; x_1+x_2+x_3 \le 1 \}$, where
$\kappa_{1}=\kappa_{2}=\kappa_{3}=\kappa_{4}=1$. The objective of this example is to evaluate the approximation quality of GPDD approximations in terms of the second-moment statistics of $y(\mathbf{X})$ and contrast the GPDD results with those obtained from the GPCE approximations.

Under Assumption 1, bases comprising multivariate orthogonal polynomials consistent with the Dirichlet probability density function exist. One such basis, obtained using a Rodrigues-type formula \cite{dunkl14} and subsequent scaling, leads to the standardized version $\{\Psi_{u,\mathbf{j}_u}(\mathbf{x}_u)\}$, as described in \eqref{4.5b}.  More explicitly, Table \ref{table1} presents first-, second-, and third-order (-degree) orthogonal polynomials in $\mathbf{x}_u$, $1 \le |u| \le 3$, obtained for the Dirichlet density function.

\begin{table}[htbp!]
\vspace{-0.1cm}
\caption{A few orthogonal polynomials consistent with the Dirichlet density function of Example 1.$^{(\text{a})}$}
\begin{centering}
\begin{tabular}{c}
\hline
\noalign{\vskip0.05cm}
\tabularnewline
{\scriptsize{}}%
\begin{tabular}{l}
{\scriptsize{}$\Psi_{\{i\}1}={\displaystyle \sqrt{\frac{7}{3}}-4\sqrt{\frac{7}{3}}x_{i}},$}\tabularnewline
{\scriptsize{}$\Psi_{\{i\}2}={\displaystyle \frac{224x_{i}^{2}}{\sqrt{55}}-28\sqrt{\frac{5}{11}}x_{i}+3\sqrt{\frac{5}{11}},}$}\tabularnewline
{\scriptsize{}$\Psi_{\{i\}3}={\displaystyle -128\sqrt{\frac{15}{13}}x_{i}^{3}+672\sqrt{\frac{3}{65}}x_{i}^{2}-112\sqrt{\frac{5}{39}}x_{i}+7\sqrt{\frac{5}{39}}},$}\tabularnewline
{\scriptsize{}$\Psi_{\{i_{1},i_{2}\}11}={\displaystyle 11\sqrt{\frac{42}{19}}x_{i_1}^{2}+2\sqrt{798}x_{i_2}x_{i_1}-14\sqrt{\frac{42}{19}}x_{i_1}+11\sqrt{\frac{42}{19}}x_{i_2}^{2}-14\sqrt{\frac{42}{19}}x_{i_2}+3\sqrt{\frac{42}{19}}},$}\tabularnewline
{\scriptsize{}$\Psi_{\{i_{1},i_{2}\}12}=\begin{array}[t]{l}
{\displaystyle -6\sqrt{\frac{1001}{37}}x_{i_{1}}^{3}-614\sqrt{\frac{77}{481}}x_{i_{2}}x_{i_{1}}^{2}+174\sqrt{\frac{77}{481}}x_{i_{1}}^{2}-2\sqrt{37037}x_{i_{2}}^{2}x_{i_{1}}+788\sqrt{\frac{77}{481}}x_{i_{2}}x_{i_{1}}}\\
{\displaystyle -114\sqrt{\frac{77}{481}}x_{i_{1}}-18\sqrt{\frac{1001}{37}}x_{i_{2}}^{3}+30\sqrt{\frac{1001}{37}}x_{i_{2}}^{2}-174\sqrt{\frac{77}{481}}x_{i_{2}}+18\sqrt{\frac{77}{481}},}
\end{array}$}\tabularnewline
{\scriptsize{}$\Psi_{\{i_{1},i_{2}\}21}=\begin{array}[t]{l}
{\displaystyle -18\sqrt{\frac{1001}{37}}x_{i_{1}}^{3}-2\sqrt{37037}x_{i_{2}}x_{i_{1}}^{2}+30\sqrt{\frac{1001}{37}}x_{i_{1}}^{2}-614\sqrt{\frac{77}{481}}x_{i_{2}}^{2}x_{i_{1}}+788\sqrt{\frac{77}{481}}x_{i_{2}}x_{i_{1}}}\\
{\displaystyle -174\sqrt{\frac{77}{481}}x_{i_{1}}-6\sqrt{\frac{1001}{37}}x_{i_{2}}^{3}+174\sqrt{\frac{77}{481}}x_{i_{2}}^{2}-114\sqrt{\frac{77}{481}}x_{i_{2}}+18\sqrt{\frac{77}{481}},}
\end{array}$}\tabularnewline
{\scriptsize{}$\Psi_{\{1,2,3\}111}=\begin{array}[t]{l}
{\displaystyle -12\sqrt{55}x_{1}^{3}-50\sqrt{55}x_{2}x_{1}^{2}-50\sqrt{55}x_{3}x_{1}^{2}+138\sqrt{\frac{11}{5}}x_{1}^{2}-50\sqrt{55}x_{2}^{2}x_{1}-50\sqrt{55}x_{3}^{2}x_{1}}\\
{\displaystyle +346\sqrt{\frac{11}{5}}x_{2}x_{1}}-{\displaystyle 128\sqrt{55}x_{2}x_{3}x_{1}+346\sqrt{\frac{11}{5}}x_{3}x_{1}-96\sqrt{\frac{11}{5}}x_{1}-12\sqrt{55}x_{2}^{3}}\\
{\displaystyle {\displaystyle -12\sqrt{55}x_{3}^{3}}+138\sqrt{\frac{11}{5}}x_{2}^{2}}-50\sqrt{55}x_{2}x_{3}^{2}+{\displaystyle 138\sqrt{\frac{11}{5}}x_{3}^{2}-96\sqrt{\frac{11}{5}}x_{2}-50\sqrt{55}x_{2}^{2}x_{3}}\\
{\displaystyle +346\sqrt{\frac{11}{5}}x_{2}x_{3}-96\sqrt{\frac{11}{5}}x_{3}+18\sqrt{\frac{11}{5}}.}
\end{array}$}\tabularnewline
\end{tabular}\tabularnewline
\noalign{\vskip0.05cm}
\hline
\end{tabular}
\par\end{centering}
{\footnotesize ~~~~~~~~~~~~~~(a) Here, $i=1,2,3$; $i_1,i_2 = 1,2,3$, $i_2>i_1$.}
\label{table1}
\end{table}

Define two relative errors
\begin{equation}
e_{S,m} :=
\displaystyle
\frac{\left|\text{var}[y(\mathbf{X})] - \text{var}[y_{S,m}(\mathbf{X})]\right|}
{\text{var}[y(\mathbf{X})]}
~~\text{and}~~
e_{p} :=
\displaystyle
\frac{\left|\text{var}[y(\mathbf{X})] - \text{var}[y_{p}(\mathbf{X})]\right|}
{\text{var}[y(\mathbf{X})]}
\label{6.2}
\end{equation}
in the variances, committed by the $S$-variate, $m$th-order GPDD approximation $y_{S,m}(\mathbf{X})$ and the $p$th-order GPCE approximation $y_{p}(\mathbf{X})$, respectively, of $y(\mathbf{X})$.  Here, the exact variance $\text{var}[y(\mathbf{X})]$ and the GPDD variance $\text{var}[y_{S,m}(\mathbf{X})]$, given $S$ and $m$, were determined analytically from their definitions, which is possible as (1) $y$ and $y_{S,m}$ are both polynomials and (2) expectations of monomials $\{ \mathbf{X}^{\mathbf{j}}, 0 \le |\mathbf{j}| < \infty \}$ for $\mathbf{X}$ following a Dirichlet distribution are known analytically. In contrast, the GPCE variance, given $p$, was calculated using the analytical formula from a prior work \cite{rahman18b}.  Therefore, all errors were calculated exactly.

Table \ref{table2} presents the errors $e_{S,m}$ and $e_{p}$, obtained using various combinations of the truncations parameters of GPDD and GPCE: $S=1,2$, $m=1,2,3,4,5$, and $p=1,2,3,4,5$.  The two truncations with respect to the degree of interaction $S=1$ and $S=2$ represent the univariate GPDD and bivariate GPDD approximations, respectively.  According to Table 1, the GPDD approximation errors drop with respect to $S$ and $m$ as expected. With the exception of $m=5$, the errors from the univariate and bivariate GPDD approximations are nearly identical.  Moreover, the errors from the two GPDD approximations are the same or very close to the errors from the respective GPCE approximations.  This is because the chosen function $y$, albeit it is highly nonlinear with respect to $\mathbf{X}$, is endowed with little interactions among input variables.  From the comparisons of computational efforts, measured in terms of the numbers of expansion coefficients also listed in Table \ref{table2}, both GPDD approximations are more efficient than the GPCE approximations for the same expansion order.  For instance, the univariate, fifth-order GPDD approximation ($S=1$, $m=5$) achieves a relative error of $3.31864\times 10^{-5}$ employing only 16 expansion coefficients. In contrast, to match the same-order error, the fifth-order GPCE approximation ($p=5$) is needed, committing a relative error of $1.69589\times 10^{-5}$ at the cost of 56 expansion coefficients.  Therefore, the univariate GPDD approximation is substantially more economical than the GPCE approximation for a similar accuracy.  The bivariate, fifth-order GPDD approximation produces practically the same result of the fifth-order GPCE approximation, but still upholding some computational advantage over the latter. However, the gain in efficiency from the bivariate GPDD approximation is much less than that from the univariate GPDD approximation.  This is expected due to the added computational expense to include, in addition to the main effects, all two-variable interaction effects, in the bivariate approximation.  Nonetheless, when the main effects of the input variables on $y$ are dominant over their interactive effects, as is the case in this example, the GPDD approximation is expected to be more effective than the GPCE approximation.

\begin{table}[htbp!]
\vspace{-0.1cm}
\caption{Relative errors in the variances of $y(\mathbf{X})$ calculated by GPDD and GPCE approximations in Example 1}
\begin{centering}
\begin{tabular}{cccccccccc}
\hline
 &  & \multicolumn{2}{c}{{\small{}Univariate GPDD}} &  & \multicolumn{2}{c}{{\small{}Bivariate GPDD}} &  & \multicolumn{2}{c}{{\small{}GPCE}}\tabularnewline
\cline{3-4} \cline{6-7} \cline{9-10}
{\small{}$m$ or $p$} &  & {\small{}$e_{1,m}$} & {\small{}$L_{1,m}$} &  & {\small{}$e_{2,m}$} & {\small{}$L_{2,m}$} &  & {\small{}$e_{p}$} & {\small{}$L_{p}$}\tabularnewline
\hline
{\small{}1} &  & {\small{}0.856363} & {\small{}4} &  &  &  &  & {\small{}0.856363} & {\small{}4}\tabularnewline
{\small{}2} &  & {\small{}0.219054} & {\small{}7} &  & {\small{}0.219038} & {\small{}10} &  & {\small{}0.219038} & {\small{}10}\tabularnewline
{\small{}3} &  & {\small{}0.038876} & {\small{}10} &  & {\small{}0.038860} & {\small{}19} &  & {\small{}0.038860} & {\small{}20}\tabularnewline
{\small{}4} &  & {\small{}$1.62697\times10^{-3}$} & {\small{}13} &  & {\small{}$1.61074\times10^{-3}$} & {\small{}31} &  & {\small{}$1.61074\times10^{-3}$} & {\small{}35}\tabularnewline
{\small{}5} &  & {\small{}$3.31864\times10^{-5}$} & {\small{}16} &  & {\small{}$1.69589\times10^{-5}$} & {\small{}46} &  & {\small{}$1.69589\times10^{-5}$} & {\small{}56}\tabularnewline
\hline
\end{tabular}
\par\end{centering}
\label{table2}
\end{table}

For a more general discussion on the computational efforts by the two aforementioned approximations, consider the respective numbers of Fourier coefficients involved: (1) $L_{S,m}$ in \eqref{5.15} for the $S$-variate, $m$th-order GPDD approximation; and (2) $L_p=(N+p)!/(N!p!)$ for the $p$th-order GPCE approximations \cite{rahman18b}.  In other words, $L_{S,m}$ grows $S$-degree polynomially with respect to $N$, whereas $L_{p}$ scales $p$-degree polynomially with $N$. For stochastic problems entailing highly nonlinear functions but containing mostly low-variate interactive effects of input variables, $p$ is expected to be much larger than $S$.  Consequently, the GPDD approximation should offer a hefty computational benefit over the GPDD approximation for the same expansion order.  As an example, consider a stochastic problem involving 20 input random variables ($N=20$) and the following truncation parameters of GPDD and GPCE: $S=1,2$, $m=4$, and $p=4$. In this case, the univariate, fourth-order GPDD approximation ($S=1$, $m=4$), the bivariate, fourth-order GPDD approximation ($S=2$, $m=4$), and the fourth-order GPCE approximation ($p=4$) require 81, 1221, and 10,626 Fourier coefficients, respectively.  Clearly, the growth of the number of Fourier coefficients in GPCE is much sharper than that in GPDD. This is primarily because a GPCE approximation is solely dictated by a single truncation parameter $p$, which controls the largest polynomial expansion order preserved, but not the degree of interaction independently.  In contrast, two different truncation parameters $S$ and $m$ are involved in a GPDD approximation, affording a greater flexibility in retaining the largest degree of interaction and largest polynomial expansion order. In consequence, the numbers of Fourier coefficients and hence the computational efforts by the GPDD and GPCE approximations can vary appreciably. A computational study comparing the accuracy and efficiency of GPDD and GPCE approximations in solving high-dimensional stochastic problems is desirable.

\section{Conclusion}
A new generalized PDD, referred to as GPDD, of a square-integrable output random variable, comprising hierarchically ordered multivariate orthogonal polynomials in dependent input random variables with non-product-type probability measures, is presented.  A dimension-wise splitting of appropriate polynomial spaces into subspaces, each spanned by measure-consistent orthogonal polynomials, was constructed, resulting in a polynomial refinement of the generalized ADD and eventually GPDD without the need for a tensor-product structure.  Under prescribed assumptions, the set of measure-consistent orthogonal polynomials was proved to form a complete basis of each subspace, leading to a union-sum collection of such sets of basis functions, including the constant subspace, to span the space of all polynomials.  In addition, the aforementioned collection is dense in a Hilbert space of square-integrable functions, leading to the mean-square convergence of GPDD to the correct limit, including when there are infinitely many random variables.  New results determining statistical properties of random orthogonal polynomials were derived.  The optimality of GPDD and the approximation quality due to truncation were demonstrated.  For independent probability measures, the proposed PDD reduces to the existing PDD, justifying the appellation GPDD introduced in this work.  By exploiting the hierarchical structure of a function, if it exists, the GPDD approximation is anticipated to solve efficiently high-dimensional stochastic problems in the presence of dependent random variables.

\bibliographystyle{elsarticle-num}
\section*{References}
\bibliography{gpdd1_rahman}

\end{document}